\documentclass[reqno, 12pt]{amsart}
\usepackage{graphicx}
\usepackage{amsmath,amssymb,amscd,mathtools}
\usepackage[centering,textwidth=6.5in,textheight=9in,marginpar=0.7in]{geometry}
\usepackage[scr]{rsfso}

\usepackage{textcomp}

\usepackage{tikz}
\usetikzlibrary{math,angles,quotes,intersections,calc,through}

\usepackage[plainpages=false,colorlinks,hyperindex,pdfpagemode=None,bookmarksopen,linkcolor=blue,citecolor=blue,urlcolor=blue]{hyperref}

\newtheorem{thm}{Theorem}[section]
\newtheorem{cor}[thm]{Corollary}

\newtheorem{pro}[thm]{Proposition}

\theoremstyle{definition}
\newtheorem{definition}[thm]{Definition}
\newtheorem{remark}[thm]{Remark}
\newtheorem{example}[thm]{Example}

\numberwithin{equation}{section}

\newcommand{\ta}{\tilde{a}}
\newcommand{\tb}{\tilde{b}}

\newcommand{\bC}{\mathbb{C}}

\newcommand{\bR}{\mathbb{R}}

\newcommand{\cE}{\mathcal{E}}

\newcommand{\cH}{\mathcal{H}}

\newcommand{\cO}{\mathcal{O}}
\newcommand{\cP}{\mathcal{P}}

\newcommand{\cW}{\mathcal{W}}
\newcommand{\cX}{\mathcal{X}}
\newcommand{\cY}{\mathcal{Y}}
\newcommand{\cZ}{\mathcal{Z}}

\newcommand{\sE}{\mathsf{E}}
\newcommand{\sF}{\mathsf{F}}

\newcommand{\eps}{\epsilon}

\newcommand{\pa}{\partial}

\newcommand{\ds}{\displaystyle}

\begin{document}

\title[Birkhoff Normal Form and Twist coefficients]{Birkhoff Normal Form and Twist Coefficients of Periodic Orbits of Billiards}

\author[X. Jin]{Xin Jin}
\address{Math Department, Boston College, Chestnut Hill, MA 02467.}
\email{xin.jin@bc.edu}
\thanks{X. J. is supported in part by  the NSF Grant DMS-1854232.}

\author[P. Zhang]{Pengfei Zhang}
\address{Department of Mathematics, University of Oklahoma, Norman, OK 73019.}
\email{pengfei.zhang@ou.edu}



\begin{abstract}
In this paper we study the Birkhoff Normal Form around elliptic periodic points
for a variety of dynamical billiards. We give an explicit construction
of the Birkhoff transformation and obtain explicit formulas
for the first two twist coefficients in terms of the geometric parameters of the billiard table. 
As an application, we obtain  characterizations of  the nonlinear
stability and local analytic integrability of the billiards around the  elliptic periodic points.
\end{abstract}

\maketitle

\section{Introduction}

Let $U\subseteq \bR^2$ be an open neighborhood of the origin $P=(0,0) \in \bR^2$,
$f:U\to \bR^2$ be  a symplectic embedding (area-preserving and orientation-preserving)
with $f(P)=P$. Let $\lambda$ and $\lambda^{-1}$ be the eigenvalues of the tangent 
matrix $D_{P}f$. Then the fixed point $P$ is said to be 
hyperbolic if $|\lambda|\neq 1$, parabolic if $\lambda=\pm 1$, and  elliptic 
if $|\lambda|=1$ and $\lambda\neq \pm 1$.
An elliptic fixed point $P$ is said to be non-resonant
if $\lambda^n \neq 1$ for any $n \ge 3$.
If $P$ is non-resonant, then for any $N\ge 1$, there exist an open neighborhood
$U_N\subset U$ of $P$ and a symplectic 
transformation
\begin{align}
h_N: U_N \to \bR^2, \begin{bmatrix} x \\ y \end{bmatrix}
\mapsto 
\begin{bmatrix} x+p_2(x,y) + \cdots + p_{2N+1}(x,y) \\ y+ q_2(x,y) + \cdots + q_{2N+1}(x,y) \end{bmatrix}
+O(r^{2N+2}), \label{formal}
\end{align}
where $p_n$ and $q_n$ are polynomials of degree $n$ for each $2\le n\le 2N+1$, 
such that
\begin{align}
h_N^{-1}\circ f\circ  h_N \Big(\begin{bmatrix} x \\ y \end{bmatrix} \Big)
=\begin{bmatrix} \cos \Theta(r^2) & -\sin\Theta(r^2) \\ \sin \Theta(r^2) & \cos\Theta(r^2)\end{bmatrix}\begin{bmatrix} x \\ y \end{bmatrix}
+ O(r^{2N+2}), \label{BNF}
\end{align} 
where  $r^2=x^2+y^2$,
$\Theta(r^2)=\theta+\tau_1 r^2+\tau_2 r^4+\cdots +\tau_N r^{2N}$,
and $\theta$ is  the argument of the eigenvalue $\lambda$.
The function $\Theta(r^2)$ measures the amount of rotations 
of points around the fixed point $P$. 
The symplectic transformation \eqref{formal} is called Birkhoff transformation, 
the resulting form \eqref{BNF} is called Birkhoff Normal Form,
and the coefficient $\tau_k$ in the function $\Theta(r^2)$  is called
the $k$-th twist coefficient of $f$ at $P$ for each $k\ge 1$.
In fact,  the above Birkhoff Normalization works
as long as the elliptic fixed point $P$ is not $(2N+2)$-resonant for some $N\ge 1$. 
That is, $\lambda^{n}\neq 1$ for each $3\le n \le 2N+2$.
The same results hold for elliptic periodic orbits.
See \cite{SiMo} for more details.
Normal form also exists for hyperbolic periodic points, see \cite{Mos56} for more details.

The existence of  Birkhoff Normal Form 
plays an important role in the study of nonlinear stability 
and local integrability of elliptic fixed points. 
Recall that a fixed point $P$ of a map $f: U \to \bR^2$ is said to be nonlinearly stable
if there is a nesting sequence of $f$-invariant neighborhoods $U_n$ of $P$, $n\ge 1$
whose boundaries $S_n =\pa U_n$ are invariant circles,
such that $\bigcap_{n\ge 1}U_n=\{P\}$.
It is said to be locally smoothly/analytically integrable 
if there is an open neighborhood $V$ of $P$ 
that is foliated smoothly/analytically into invariant circles around $P$.
An elliptic fixed point being nonresonant alone does not guarantee it is nonlinearly 
stable.
In fact, Anosov and Katok \cite{AnKa} 
constructed an ergodic symplectic diffeomorphism $f$ of 
the closed unit disk $\bar{D}$ with a nonresonant elliptic 
fixed point at the origin $P\in \bar{D}$. On the other hand,
Moser's Twist Mapping Theorem \cite[Theorem 2.13]{Mos73}  states that
if the elliptic fixed point $P$ is not $(2N+2)$-resonant, 
then  the fixed point $P$ is nonlinearly stable
as long as one of its first $N$ twist coefficients is nonzero\footnote{
It is noteworthy that an elliptic fixed point is nonlinearly stable
if it satisfies stronger arithmetic conditions such as the Bruno condition
and the Diophantine condition, 
see \cite{Rus02, FK09} for more details.}.

There have been many results about the twist coefficients 
(especially the first twist coefficient $\tau_1$) at elliptic periodic points.
Meyer \cite{Mey70} studied the generic bifurcation
of periodic points using the property that the first twist coefficient is nonzero
for generic irrational elliptic periodic points.
Meyer \cite{Mey71} studied the generic stability properties
of periodic points using the property that at least one of the first two twist coefficients is nonzero
for a generic family of irrational elliptic periodic points.
Moeckel \cite{Moe90} studied the behavior of the first twist coefficient $\tau_1(\phi_a, P)$ 
of an elliptic fixed point $P$
for a generic one-parameter family $(\phi_a)$ of area-preserving diffeomorphisms
and characterized the orders of the poles of $\tau_1(\phi_a, P)$ at resonances.
See also \cite{Tor02} for an application to a Hill's equation with singular term. 
The twisting properties also have applications in dynamical billiards.
Dias Carneiro, Kamphorst and Pinto de Carvalho \cite{KP01, DKP03} 
considered the periodic orbit $\cO_2=\{P, F(P)\}$
of period 2 of dynamical billiards and obtained a formula for  $\tau_1(F^2, P)$, 
where $F$ is the (locally defined) billiard map. 
Moreover, they proved the existence of an elliptic island by a small normal perturbation 
of the billiard table along an elliptic periodic orbit of period $2$.
The same result for elliptic periodic orbits of higher periods is proved by Bunimovich and Grigo
\cite{BuGr}.
An explicit formula of $\tau_1(F^2, P)$ in terms of the geometric terms of the billiard table
is given in \cite{KP05}.

In this paper we study the  Birkhoff transformation and the first two twist coefficients of 
certain periodic orbits of (locally defined) billiards.
Let $a(t)=\sum_{n\ge 1}a_{2n} t^{2n}$ and $b(t)=\sum_{n\ge 1} b_{2n} t^{2n}$, 
$t\in (-\eps, \eps)$  for some $\eps>0$, be two even functions,
$\gamma_0(t)=(a(t), t)$ and $\gamma_1(t)=(L -b(t), -t)$
be two parametrized curves on the plane $\bR^2$, see Fig.~\ref{btable}. 
Denote by $Q(L, a, b)$,   the domain bounded by the two curves $\gamma_j$, $j=0 ,1$, 
with their endpoints connected by two horizontal segments.
We will  consider the dynamical billiards on the domain  $Q(L, a, b)$.
Let $M$ be the phase space of the billiards, and $F: M\to M$ be the billiard map
that sends a light ray on the billiard table $Q(L,a,b)$ to the reflected ray upon hitting the boundary of the table. See \cite{ChMa} for more details about dynamical billiards. 
On the table $Q(L, a, b)$, there exists a periodic orbit of period $2$ 
bouncing back and forth between the two points $\gamma_0(0)$ and $\gamma_1(0)$.
Denote this orbit by $\cO_2=\{P,F(P)\}$,
where $P$ is the point with impact point at $\gamma_0(0)$.

\begin{remark}
The assumption of the two functions $a(t)$ and $b(t)$ being even simplifies our computations
for the two twist coefficients. Without this assumption, the twist coefficient $\tau_1$ will depend on both $R_j'(0)$ and $R_j''(0)$, $j=0, 1$, (see \cite{KP05} for example), and the twist coefficient $\tau_2$ will depend on $R_j^{(k)}(0)$ for all $0\le k \le 4$, $j=0, 1$. To have a taste of the  twist coefficient $\tau_2$ in the general case, see \cite{CGM}, in which the authors obtained $\cP_3(F)$, a quantity needed to compute $\tau_2$, that occupies three and a half pages of \cite{CGM}.
\end{remark}

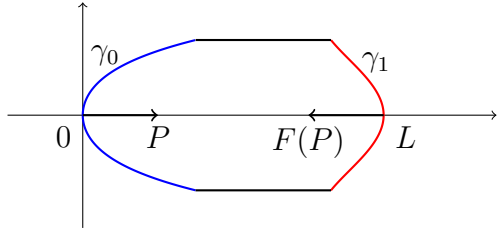
\begin{figure}[htbp]
\begin{tikzpicture}
\draw[->] (-1,0) -- (5.5,0);
\draw[->] (0,-1.5) -- (0,1.5);
\draw [domain=-1:1, samples=100, red, thick] plot ({4-\x*\x +0.3*\x*\x*\x*\x }, {\x});
\node at (3.9, 0.7) {$\gamma_1$};  
\draw [domain=-1:1, samples=100, blue, thick] plot ({\x*\x +0.5*\x*\x*\x*\x }, {\x});
\node at (0.3,0.8) {$\gamma_0$};
\draw[thick, ->] (0,0) node[below left]{$0$} -- (1,0) node[below]{$P$};
\draw[thick, ->] (4,0) node[below right]{$L$} -- (3,0) node[below]{$F(P)$};
\draw[thick] (3.3, 1) -- (1.5, 1) (3.3, -1) -- (1.5, -1);
\end{tikzpicture}
\caption{$\gamma_0(t)= (t^2 +0.5 t^4, t)$ and $\gamma_1(t)=(4 -  t^2 +0.3 t^4, -t)$, 
$t\in(-1, 1)$.}
\label{btable}
\end{figure}

Let  $s=s(t)$ be the arc-length parameter of the curve $\gamma_j$ with $s(0)=0$, 
$R_j(s)$ be the radius of curvature of the curve $\gamma_j$ at $\gamma_j(s)$,
and $R^{(k)}_j(s)$ be the $k$-th derivative of $R_j(s)$ with respect to $s$, $j\ge 0$. 
The quantity $L$ represents the length of the trajectory between two consecutive impacts
of the periodic orbit $\cO_2$. 
If we enlarge the table $Q=Q(L, a, b)$ by $k$ times and denote the scaled table by $kQ$, 
then the corresponding quantities $L$ and $R_j$, $j=0,1$, are enlarged by $k$ times.
In some sense, this scaling just adds a pair of glasses for us to observe 
 the dynamics on the same billiard table.
In particular, one can assume that $L=1$ without losing much generality. 
We will keep using $L$ and determine how the twist coefficients change 
with respect to the scaling from $Q$ to $kQ$. 
For convenience we introduce a short notation related to scaling the billiard table:
\begin{definition} \label{def.deg}
A quantity $\tau(Q)$ of the billiard table $Q$ is said to be homogeneous 
with respect to the scaling if there exists a constant $d\in \bR$
such that $\tau(kQ)=k^d \tau(Q)$. 
Denote it by $\deg(\tau)=d$, which is called the scaling degree of the quantity $\tau(Q)$. 
\end{definition}
It is easy to see that $\deg(L)=1$ and $\deg(R_{j}^{(k)}(0))=1-k$ for any $k\ge 0$, $j=0,1$. 
As we will see later, the twist coefficient $\tau_j$ is homogeneous 
with degree $\deg(\tau_j)=-j$, $j=1,2$.

\subsection{Symmetric billiards}
We  start with a special case with $R_0(s)=R_1(s)=:R(s)$.
In this case it is clear that $\tau_k(F^2, P)=2\tau_k(F,P)$ for all $k\ge 1$,
see \S\ref{sec.taylor} for more details.
We use the short notation that $R^{(k)}=R^{(k)}(0)$ for each $k\ge 0$.
In particular, the tangent matrix of the one-step map $F$ at the periodic point $P$ is given by 
\begin{align}
D_{P}F=\begin{bmatrix} 
\frac{L}{R} -1 &  L \\ 
\frac{L}{R^2} -\frac{2}{R} &  \frac{L}{R}-1
\end{bmatrix}.\label{DF}
\end{align}
It follows that the matrix $D_{P}F$ is hyperbolic if $\frac{L}{R}>2$,
is parabolic if $\frac{L}{R}=2$, and is elliptic if $0< \frac{L}{R}< 2$.

In the following we will assume  $0<\frac{L}{R}<2$ and say that the point $P$ is elliptic.
Note that the periodic orbit $\cO_2=\{P, F(P)\}$ is in fact parabolic when $\frac{L}{R}=1$.
The eigenvalues of the tangent matrix $D_{P}F$ at the periodic point $P$
are $\lambda$ and $\bar\lambda$, where
$\lambda = \frac{L}{R}-1 - i \sqrt{L(\frac{2}{R} -\frac{L}{R^2})}$. 
We will need the following nonresonance assumptions:
\begin{enumerate}
\item[(A1)] $\lambda^4 \neq 1$, or equally,
$\frac{L}{R}\in(0, 2)\backslash\{1\}$;

\item[(A2)] $\lambda^6 \neq 1$, or equally,
$\frac{L}{R}\in(0, 2)\backslash\{\frac{1}{2},\frac{3}{2}\}$.
\end{enumerate}
The assumption (A1) is needed for obtaining the normal form 
containing the first twist coefficient $\tau_1$
and both (A1) and (A2) are needed for obtaining the normal form containing the 
first two twist coefficients  $\tau_1$ and $\tau_2$.
\begin{thm}
Assuming (A1), 
the first twist coefficient $\tau_1(F,P)$ of the one-step billiard map $F$ at $P$ is  given by
\begin{align}
\tau_1(F,P)= &\frac{1}{8R} - \frac{L}{8(2R-L)}R''. \label{tau1.sym.F}
\end{align}
Assuming (A1) and (A2), 
the second twist coefficient $\tau_2(F,P)$ of the one-step billiard map $F$ at $P$ is given by
\begin{align}
\tau_2(F,P)=
&\frac{1}{64}\cdot \Big(\frac{3 (7 R^2 - 8 R L + 2 L^2)}{4 R^2 (R -  L) \sqrt{(2R-L)L}}
-\frac{\sqrt{L}(27 R^2 - 40 R L +  10 L^2)}{6 R (R - L) (2R - L)^{3/2}} R'' \nonumber \\
&+ \frac{L^{3/2} (31 R^2 - 36 R L +  6 L^2) }{12 (R - L) (2R - L)^{5/2}}(R'')^2
-\frac{L^{3/2} R}{3 (2R - L)^{3/2}}R^{(4)} \Big). \label{tau2.sym.F}
\end{align}
\end{thm}

\begin{remark}
Consider the dynamical billiards inside the ellipse 
$\frac{x^2}{b^2}+y^2 =1$, where $0<b<1$. 
It is a classical result that the elliptic billiard is completely integrable: 
the whole phase space  is foliated  into invariant curves.
See \cite[\S1.4]{ChMa}.
Moreover, the periodic orbit $\cO_2$ along the minor axis ($x$-axis) of the ellipse
is elliptic and 
is surrounded by $F^2$-invariant circles, 
on which the restriction of $F^2$ acts as circle diffeomorphisms.
So we can define the rotation number of $F^2$ on each of these invariant circles. 
An exact formula of the rotation number was obtained by Ko{\l}odziej \cite{Kol85},
see also \S\ref{sec.Kol}. 
Tabanov \cite{Tab94} obtained formulas of the rotation numbers using
the action-angle coordinates of elliptic billiards.
On the other hand,
one can get an approximation of the rotation number 
for the invariant circles around the orbit $\cO_2$
using the twist coefficients $\tau_1$ and $\tau_2$.
In \S \ref{sec.Kol} we show that this approximation is as accurate as it can be: 
it matches Ko{\l}odziej's formula  up to the 4th-order.
\end{remark}

\subsection{Applications} \label{tau1.app}
In this subsection we give some applications of the Birkhoff Normal Form
for the periodic orbit $\cO_2$ on a variety of  symmetric billiards.
We first introduce a few quantities that are needed below
and postpone the general discussion to \S \ref{sec.tau1}.
Let $z$ and $w$ be the complex coordinates that span the eigenspaces of the eigenvalues $\lambda$ and $\bar\lambda$ of the 
tangent matrix $D_P F$, respectively,
and $c_{jk}$, $j \ge 0$, $k\ge 0$, be the coefficients of the Taylor series of the 
billiard map $F$ at $(z,w)=(0,0)$, see \S \ref{tau1.complex}.
Let $d_{jk}$, $j+k=3$, be the coefficients of the first-step Birkhoff transformation
in the complex coordinate $(z,w)$, see \S \ref{tau1.intermediate}.
For example, $d_{03}=\frac{c_{03}}{1-\overline{\lambda}^4}$.
Therefore, in the resonance case $\lambda^4=1$,  $c_{03}\neq 0$
is an obstruction for the existence of  Birkhoff transformation for $F$ at the elliptic point $P$, see Remark \ref{loc.ana.int}.

\begin{example}\label{ex.ell.tau1}
Let $0<b<1$. Consider  the  billiards inside the ellipse 
$\frac{x^2}{b^2}+y^2 =1$ 
and the periodic orbit $\cO_2=\{P, F(P)\}$ along the minor axis ($x$-axis) of the ellipse. 
Then $L=2b$, $R=\frac{1}{b}$, and  $P$ is an elliptic point for the billiard map 
with an eigenvalue $\lambda=2 b^2 -1- i 2b\sqrt{1-b^2}$.
The nonresonance assumption (A1) holds for $b\in (0,1)\backslash\{\frac{1}{\sqrt{2}}\}$.
Since  $R''=\frac{3(b^2-1)}{b}$, we get  $\tau_1(F,P)= \frac{b}{2}$ from \eqref{tau1.sym.F}.
It follows from  Moser's Twist Mapping Theorem
that the periodic orbit $\cO_2$ is nonlinearly stable.
This is indeed consistent with the fact that the elliptic billiard is integrable.
\end{example}

\begin{remark}\label{ellipse.resonance}
Consider the resonance case $b=\frac{1}{\sqrt{2}}$ of the elliptic billiards that is excluded 
from Example \ref{ex.ell.tau1}.
Note that the coefficient $c_{03}(b)$  is given by
\begin{align}
c_{03}(b)=-\frac{\bar{\lambda}}{6}b (1-2b^2)(3- 4b^2)(2b\sqrt{1-b^2} +i(1-2b^2)).
\end{align}
In particular, $c_{03}(b)$ vanishes at $b=\frac{1}{\sqrt{2}}$. 
It follows that there is no obstruction for the existence of the first-step Birkhoff transformation.
One can simply set $d_{03}(\frac{1}{\sqrt{2}})=0$. This does work, 
just the corresponding Birkhoff transformation
does not depend continuously on the parameter $b$ at $b=\frac{1}{\sqrt{2}}$.
A canonical approach is to let the factor $(1-2b^2)$ of $c_{03}(b)$ 
cancel the factor  $(1-2b^2)^{-1}$ in $\frac{1}{1-\overline{\lambda}^4}$ when defining $d_{03}$.
In this way we obtain:
\begin{align}
d_{03}(b)=\frac{(-3 + 4 b^2)(1 - 8 b^2 + 8 b^4 - i 4b(1-b^2)^{1/2}(1-2b^2))^2}{48(1-b^2)^{1/2}},
\end{align}
which is well defined and continuous for all $b\in(0,1)$.
Therefore, Moser's Twist Mapping Theorem is applicable for the periodic orbit $\cO_2$ for all elliptic billiards. 
\end{remark}

\begin{example}\label{ex.lemon}
Consider the lemon billiards introduced in  \cite{HeTo}.
Recall that a lemon $Q(L)$ is defined as the domain bounded by two  circular arcs $\gamma_j$, $j=0,1$ of radius $R=1$,
where $L$ measures the distance between $\gamma_0(0)$ and $\gamma_1(0)$,
see Fig.~\ref{lemon}. 
Let $\cO_2=\{P,F(P)\}$ be the periodic orbit bouncing back and forth between $\gamma_0(0)$ and $\gamma_1(0)$. 
Then the point $P$ is elliptic for the billiard map $F$ when $L\in (0,2)$,
for which the eigenvalue is $\lambda=L-1- i \sqrt{L(2-L)}$.
It satisfies  the nonresonant assumption (A1)  when $L\in (0,2)\backslash \{1\}$.
Applying \eqref{tau1.sym.F}, we get  $\tau_1(F,P)= \frac{1}{8}$, 
and the periodic orbit $\cO_2$ is nonlinearly stable whenever $L\in(0,2)\backslash \{1\}$.
\end{example}

\begin{remark}\label{re.lemon}
Consider the resonance case $L=1$ of the lemon billiards $Q(L)$ that is excluded 
from Example \ref{ex.lemon}.  
Note that $c_{03}(1)= -\frac{i}{8}\neq 0$, which is an obstruction for 
the existence of the first-step Birkhoff transformation for the lemon billiards $Q(1)$.
Note that the point $\gamma_{j}(0)$ is the center of the other circular arc $\gamma_{1-j}$,
$j=0,1$, and all the orbits hitting $\gamma_{j}(0)$ on $\gamma_{1-j}$ are periodic points
of period $4$, see Fig.~\ref{lemon}. Therefore, there is no twist at all around the 
periodic orbit $\cO_2$.
Numerical results in \cite{CMZZ} suggest that the billiard map on $Q(1)$ is in fact ergodic.
In particular, the periodic orbit $\cO_2$ is not nonlinearly stable.
\end{remark}

\begin{figure}[htbp]
\tikzmath{
\r=2;
\a=sqrt(3);
}
\begin{tikzpicture}
\draw[domain=-60:60, thick , samples=100] plot ({\r*cos(\x)}, {\r*sin(\x)});
\draw[domain=120:240, thick, samples=100] plot ({\r+\r*cos(\x)}, {\r*sin(\x)});
\fill (0,0) node[left]{$\gamma_{0}$} circle[radius=0.05];
\fill (2,0) node[right]{$\gamma_{1}$} circle[radius=0.05];
\draw[thick] (0,0) -- (2, 0);
\draw[dashed, thick] (0,0) -- ({\r*cos(15)}, {\r*sin(15)}) node[pos=0.3](A){}; 
\draw[dashed, thick] (0,0) -- ({\r*cos(-15)}, {\r*sin(-15)}) node[pos=0.3](B){};
\draw[dashed, thick] (\r,0) --  ({\r-\r*cos(15)}, {\r*sin(15)})  node[pos=0.7](C){}; 
\draw[dashed, thick] (\r,0) -- ({\r-\r*cos(-15)}, {\r*sin(-15)}) node[pos=0.7](D){};
\draw[blue, thick,->]  (0,0) -- (A);
\draw[blue, thick,->]  (0,0) -- (B);
\draw[red, thick,->]  ({\r-\r*cos(15)}, {\r*sin(15)}) -- (C);
\draw[red, thick,->]  ({\r-\r*cos(-15)}, {\r*sin(-15)}) -- (D);
\end{tikzpicture}
\quad
\begin{tikzpicture}
\fill (0,0) node[below right]{$P$} circle[radius=0.05];
\draw[blue, thick] (0, -\a) -- (0, \a) node[left]{$u$};
\draw[red, thick] (-\a, 0) -- (\a, 0) node[below] {$s$};
\end{tikzpicture}
\caption{The lemon table $Q(1)$ with two families of periodic points of period 4.
Left: periodic points on the table; right: periodic points on the phase space.}\label{lemon}
\end{figure}
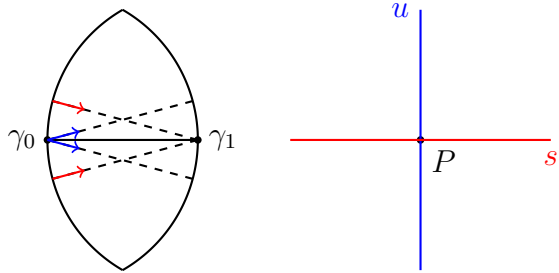

Note that the above discussion is closely related to 
the local analytic integrability of the billiard map $F$ around the periodic orbit $\cO_2$. 
Recall that
\begin{pro}\label{pro.equi}
Let $f: U\to \bR^2$ be an analytic symplectic embedding 
and $P \in U$ be an elliptic fixed point of $f$. 
Then $f$ is locally analytically integrable around $P$ if and only if 
$f$ admits an analytic Birkhoff Normal Form via 
an analytic Birkhoff transformation around the point $P$.
\end{pro} 
See \cite{Ito89, Ito92, KKN98, Zun05} for general results about local analytic integrability
about Hamiltonian systems. Generally speaking, the Birkhoff Normal Form is divergent,
see \cite{Kir19} for more details.

Let us go back to the lemon billiards with $L=1$.
Since $\lambda^4=1$ and $c_{03}\neq 0$,  
the first-step Birkhoff transformation does not exist.
In other words, the billiard map $F$ does not admit 
any analytic normalization around the point $P$.
It follows from Proposition \ref{pro.equi} that 
the lemon billiard $Q(1)$ is not locally analytically integrable around the periodic orbit $\cO_2$.

\begin{example}\label{ex.tau1.general}
Consider the general case that $\gamma_0(t)=(a(t), t)$ and $\gamma_1(t)=(1-a(t), -t)$,
where $a(t)=a_2 t^2 + a_4 t^4 + a_6 t^6 +\cdots $, $-\eps < t< \eps$.
The curvature $R(s)$ satisfies
$R=\frac{1}{2a_2}$ and $R''=6a_2 - \frac{6a_4}{a_2^2}$.
The point $P$ is elliptic when $a_2\in(0, 1)$, for which the eigenvalue is
$\lambda=2a-1-i\sqrt{a-a^2}$.
It satisfies the assumption (A1) when $a_2\in(0, 1)\backslash\{\frac{1}{2}\}$.
It follows from \eqref{tau1.sym.F} that $\tau_1(F,P)=0$ if and only if  
$R''=2-\frac{1}{R}$, or equally, 
\begin{align}
a_4= \frac{1}{3}a_2^2(4a_2 - 1).
\end{align}
Therefore, for each $a_2\in(0, 1)\backslash\{\frac{1}{2}\}$,
the periodic orbit $\cO_2$ of the billiard system is nonlinearly stable 
as long as $a_4 \neq \frac{1}{3}a_2^2(4a_2-1)$. 
\end{example}

\begin{remark}
Consider the resonance case that $a_2=\frac{1}{2}$.
Then $R=L=1$ and the coefficient $c_{03}$  is given by
\begin{align}
c_{03}=-\frac{\bar{\lambda}}{24 }(3 + R'').
\end{align}
There are two subcases:
\begin{enumerate}
\item $R''= -3$, or equally, $a_4= a_2^2(a_2 +\frac{1}{2})=\frac{1}{4}$:
there exists a first-step Birkhoff transformation
(although not unique, see the discussion in Remark \ref{ellipse.resonance}).
It follows that the  periodic orbit $\cO_2$ is nonlinearly stable 
since  $\tau_1(F,P)=\frac{1}{2}$.

\item $R'' \neq -3$, or equally, $a_4 \neq \frac{1}{4}$:
then $c_{03}\neq 0$, which is an obstruction for 
the existence of the  first-step Birkhoff transformation for the symmetric billiards.
By Proposition \ref{pro.equi}, 
the  billiard map is not locally analytically integrable around the orbit $\cO_2$.
\end{enumerate}
\end{remark}

\noindent\textbf{Example \ref{ex.tau1.general} (continued).}
We have studied the case when $\tau_1 \neq 0$.
Now we deal with the remaining case  that  $\tau_1=0$, 
or equally, $R'' = 2-\frac{1}{R}$.
Suppose $a_2\in(0, 1)\backslash\{\frac{1}{4},\frac{1}{2},\frac{3}{4}\}$.
Substituting  $R''=2-\frac{1}{R}$ into the formula \eqref{tau2.sym.F}, we get
\begin{align}
\tau_2=\frac{(2R-1)^{1/2}(5R-1)}{192(R-1)R^2}
-\frac{R}{192(2R-1)^{3/2}}R^{(4)}.
\end{align}
It follows that $\tau_2=0$ if and only if $R^{(4)}=\frac{(2R- 1)^2(5R-1)}{(R-1)R^3}$.

Note that the radius of curvature $R(s)$ of the curve $\gamma_0$  satisfies
$R^{(4)}=-\frac{12 (2 a_2^6 + 4 a_2^3 a_4 - 36 a_4^2 + 15 a_2 a_6)}{a_2^3}$.
The assumption $\tau_1=0$ implies that $a_4 = \frac{1}{3}a_2^2(4a_2-1)$.
Therefore, $\tau_2=0$ if and only if
\begin{align}
a_6=\frac{2 a_2^3 (-1 + 46 a_2 - 168 a_2^2 + 168 a_2^3)}{45(2 a_2- 1)}.
\end{align}
Collecting terms, we have
\begin{cor}
Let $a_2\in(0, 1)\backslash\{\frac{1}{4},\frac{1}{2},\frac{3}{4}\}$,
$a(t)=a_2 t^2 + a_4 t^4 + a_6 t^6 + \cdots$, and $\cO_2$
be the periodic orbit of period $2$ of the billiard map on the domain $Q(1,a, a)$.
Then the periodic orbit $\cO_2$ is nonlinearly stable
if either  $a_4 \neq \frac{1}{3}a_2^2(4a_2-1)$
or $a_6\neq \frac{2 a_2^3 (-1 + 46 a_2 - 168 a_2^2 + 168 a_2^3)}{45(2 a_2- 1)}$. 
\end{cor}
\begin{proof}
It suffices to note that either $\tau_1\neq 0$ or $\tau_2\neq 0$.
Then the result follows from Moser's Twist Mapping Theorem.
\end{proof}

\subsection{Asymmetric billiards}
Consider the general case that the function $a(t)$ for $\gamma_{0}(t)=(a(t),t)$ 
and the function $b(t)$ for $\gamma_{1}(t)=(L- b(t), -t)$ are different. 
Recall that we have set $R^{(j)}_k=R^{(j)}_k(0)$, $k=0, 1$, for each $j\ge 0$.
We will need the following nonresonance assumptions for
the iterate $F^2$ along the periodic orbit $\cO_2$:
\begin{enumerate}
\item[(B1)]  $\lambda^4\neq 1$, or equally,
$(\frac{L}{R_0}-1)(\frac{L}{R_1}-1) \in (0,1)\backslash\{\frac{1}{2}\}$;

\item[(B2)]  $\lambda^6\neq 1$, or equally,
$(\frac{L}{R_0}-1)(\frac{L}{R_1}-1) \in (0,1)\backslash\{\frac{1}{4}, \frac{3}{4}\}$.
\end{enumerate}

\begin{thm}
Assuming (B1), 
the first twist coefficient of the billiard map $F$ along the orbit $\cO_2=\{P, F(P)\}$ is given by
\begin{align}
\tau_1(F^2, P) &= \frac{1}{8} \Big(\frac{R_0+R_1}{R_0R_1}
-\frac{L}{R_0 +R_1 - L}\Big(\frac{R_1 -L}{R_0 -L}R_0'' + \frac{R_0 -L}{R_1 -L}R_1''\Big) \Big). \label{tau1.asym.F}
\end{align}
Let $k=0$ if $L<\min\{R_0, R_1\}$ and $k=1$  if $\max\{R_0, R_1\}<L< R_0+ R_1$.
Assuming (B1) and (B2),
the second twist coefficient of the billiard map $F$ along the orbit $\cO_2=\{P, F(P)\}$ is given by
\begin{align}
\tau_2
&=\frac{(-1)^k}{\sqrt{\Delta}}\Big(
\frac{3N(L,R_0,R_1)}{512R_0^2R_1^2 \Gamma}
+ \frac{P(L,R_0,R_1)(R_0'')^2+ P(L,R_1,R_0)(R_1'')^2}
{1536(R_0-L)^2(R_1 -L)^2(R_0 +R_1 -L)^2 \Gamma} \nonumber \\
&
+\frac{Q(L, R_0, R_1)R_0'' R_1''}{768 (R_{0} + R_{1}-L)^2 \Gamma}
- \frac{S(L,R_0,R_1)R_1 R_0''+ S(L,R_1,R_0) R_0 R_1''}
{768R_0R_1(R_0-L)(R_1 -L)(R_0 +R_1 -L) \Gamma} \nonumber  \\
& -\frac{T(L,R_0,R_1)(R_1 - L)R_0^{(4)}+ T(L,R_1,R_0)(R_0 -L) R_1^{(4)}}
{192(R_0-L)(R_1 -L)(R_0 +R_1 -L)}\Big), \label{tau2.asym.F}
\end{align}
where
$\Delta=L(R_0-L)(R_1 -L)(R_0 +R_1 -L)$, $\Gamma=2(R_0-L)(R_1 -L)-R_0R_1$, and
the  functions $N(L, R_0, R_1)$, $P(L, R_0, R_1)$, $Q(L, R_0, R_1)$, 
$S(L, R_0, R_1)$ and $T(L, R_0, R_1)$ are given by
\begin{align}
N(L, R_0, R_1)
=&8 L^4 (R_0^2 + R_1^2) - 16 L^3 (R_0^3 + 2 R_0^2 R_1 + 2 R_0 R_1^2 + R_1^3) \nonumber\\
&+ 8 L^2 (R_0 + R_1)^2 (R_0^2 + 4 R_0 R_1 + R_1^2) \nonumber \\
& - 8 L R_0 R_1 (2 R_0^3 + 7 R_0^2 R_1 + 7 R_0 R_1^2 + 2 R_1^3)+ 
 7 R_0^2 R_1^2 (R_0 + R_1)^2;\\
P(L, R_0, R_1)
=& L^2(R_1-L)^4 (48 R_0^3 (R_1-2L) + 24 L^2(R_1-L)^2 \nonumber \\
  & - 72L R_0(R_1-L)(R_1-2L) + R_0^2 (216L^2 - 216 R_1L + 31 R_1^2));\\
Q(L, R_0, R_1)
=&-L^2 R_{0} R_{1} (32 L^2 + 17 R_{0} R_{1} - 32 L (R_{0} + R_{1}));\\
S(L, R_0, R_1)
=&L(R_{1}-L)^2(40 ( R_1 -L)^2L^2 + 3 R_0^3 (9 R_1-16L) \nonumber \\
&  -  80 R_0 (2L^2 - 3 R_1L + R_1^2)L + 3 R_0^2 (56L^2 - 56 R_1L + 9 R_1^2));\\
T(L, R_0, R_1)=&L^2 R_{0} (R_{1}-L)^2.
\end{align}
\end{thm}

See Remark \ref{largeL} for explanations about the sign appearing in \eqref{tau2.asym.F}.

\begin{remark}
Assuming the triple $(L, R_0, R_1)$ satisfies (B1), 
the zero locus of $\tau_1$ is a line in the $(R_0'', R_1'')$-plane.
Similarly, assuming the triple $(L, R_0, R_1)$ satisfies (B1) and (B2), 
the zero locus of $\tau_2$ is a line in the $(R_0^{(4)}, R_1^{(4)})$-plane.
\end{remark}

We start with an observation that in the case  $R_0''\le 0$ and $R_1''\le 0$, 
the first twist coefficient $\tau_1$ is positive and the elliptic periodic orbit $\cO_2$ is nonlinearly stable.
\begin{example}
Consider the table consisting of two half-ellipses:
the left half given by $\gamma_0: \frac{x^2}{b_0^2}+ y^2=1$, $x\le 0$,
while the right half given by 
$\gamma_1: \frac{x^2}{b_1^2}+ y^2 =1$, $x\ge 0$,
where $0<b_1< b_0 <1$. 
Consider the periodic orbit $\cO_2$ along the $x$-axis.
Then $L=b_0+b_1$, $R_0=\frac{1}{b_0}< R_1=\frac{1}{b_1}$.
So the orbit $\cO_2$ is elliptic when either $b_0+b_1<\frac{1}{b_0}$
or $b_0+b_1>\frac{1}{b_1}$ and satisfies (A1) when 
$(b_0(b_0+b_1)-1)(b_1(b_0+b_1)-1)\neq \frac{1}{2}$. 
Then it follows from \eqref{tau1.asym.F} that $\cO_2$ is nonlinearly stable.
\end{example}

\begin{example}\label{asymL}
Consider the asymmetric lemon billiards introduced in \cite{CMZZ}.
Recall that an asymmetric lemon $Q(r, B, R)$ is the intersection of two round disks
of radii $r<R$, respectively, where $R-r<B<R+r$ measures the distance between 
their centers. The periodic orbit $\cO_2$ of period $2$ satisfies
$(L, R_0, R_1)=(R+r-B, r, R)$. This orbit is elliptic if $B<r$ or $B>R$.
Assuming  $2(B-r)(B-R)\neq rR$, the orbit
$\cO_2$ is nonlinearly stable,
since the first twist coefficient is $\tau_1(F^2, P) = \frac{1}{8}(r^{-1} + R^{-1}) \neq 0$.
This includes the table $Q(r,B,R)$ with $4(B-r)(B-R)= rR$, which is
one of the two cases left out in \cite[Proposition II.2]{CMZZ}.
See also \cite[Section 4.1]{KP05}.
\end{example}

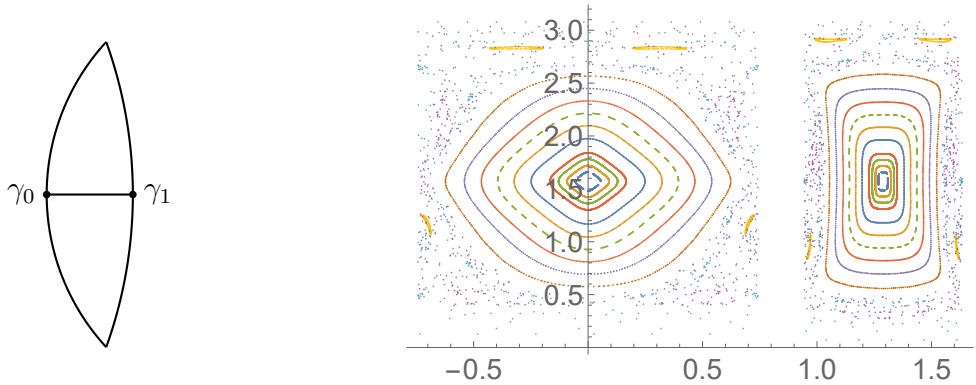
\begin{figure}[htbp]
\begin{minipage}{0.35\linewidth}
\centering
\tikzmath{
\r=3;
\b=2.618*\r;
}
\begin{tikzpicture}
\draw[domain=-0.743553:0.743553, samples=100, thick] plot ({-\r*cos(\x r)}, {-\r*sin(\x r)});
\draw[domain=-0.345273:0.345273, samples=100, thick] plot ({-\b +2*\r*cos(\x r)}, {2*\r*sin(\x r)});
\fill (-\r,0) node[left]{$\gamma_0$} circle[radius=0.05];
\fill ({2*\r -\b}, 0) node[right]{$\gamma_1$} circle[radius=0.05];
\draw[thick] (-\r,0) -- ({2*\r -\b}, 0);
\end{tikzpicture}
\end{minipage}
\quad
\begin{minipage}{0.55\linewidth}
\centering
\includegraphics[height=2in]{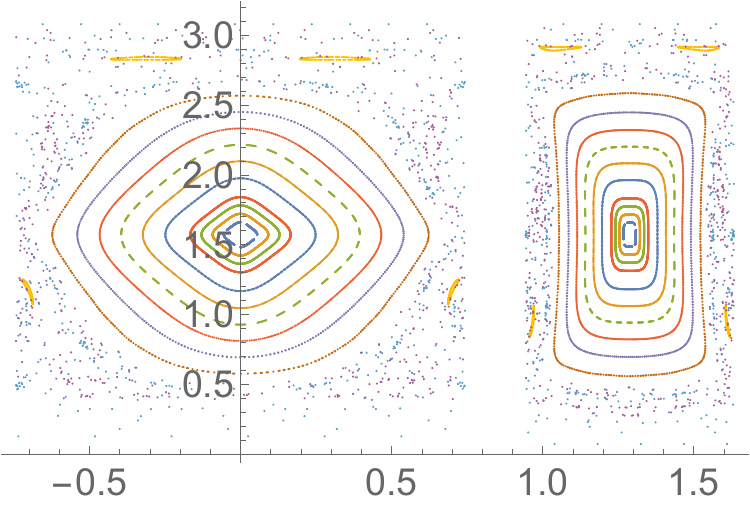}
\end{minipage}
\caption{The asymmetric lemon billiards $Q(r,B_{+}(r,R),R)$ when $R=2r$.
Left: the table; Right: the phase portrait.}\label{phaseBR}
\end{figure}

\begin{remark}
Consider the resonance case $2(B-r)(B-R)= rR$ that is excluded from 
the discussion in Example \ref{asymL}, which corresponds 
to two tables with $B_{\pm}(r,R):=\frac{1}{2}(r+R \pm \sqrt{r^2+R^2})$.
In both cases, we have $c_{03}=\frac{\bar{\lambda}}{8}(\frac{1}{r}-\frac{1}{R}) \neq 0$, which is an obstruction for  the existence of the  first-step Birkhoff transformation for the symmetric billiards.
It follows from Proposition \ref{pro.equi}
that the the asymmetric lemon billiard $Q(r,B_{\pm}(r,R), R)$ is not locally analytically integrable around the periodic orbit $\cO_2$.
Numerical results suggest that the  periodic orbit $\cO_2$ is nonlinearly stable,
see Fig.~\ref{phaseBR}.
\end{remark}

In the next subsection we will consider a new class of billiards, 
the ellipse-hyperbola lens billiards. Being asymmetric billiards, there
is a one-parameter family of such billiards  satisfying $R_0=R_1=:R$. 
For such billiards, the nonresonance assumptions (B1) and (B2) reduce to:
\begin{enumerate}
\item[(B1)$'$] $|\frac{L}{R}-1| \in (0,1)\backslash\{\frac{\sqrt{2}}{2}\}$;

\item[(B2)$'$] $|\frac{L}{R}-1| \in (0,1)\backslash\{\frac{1}{2}, \frac{\sqrt{3}}{2}\}$.
\end{enumerate}
Assuming (B1)$'$, it follows from \eqref{tau1.asym.F} that
the first twist coefficient of the billiard map $F$ along the
orbit $\cO_2=\{P, F(P)\}$ is given by
\begin{align}
\tau_1(F^2, P) &= \frac{1}{8} \Big(\frac{2}{R}
-\frac{L}{2R - L}(R_0'' +R_1'') \Big). \label{tau1.asym.R}
\end{align}
Similarly, assuming (B1)$'$ and (B2)$'$, it follows from \eqref{tau2.asym.F} that
the second twist coefficient of the billiard map $F$ along the
orbit $\cO_2=\{P, F(P)\}$ is given by
\begin{align}
\tau_2
&=
\frac{3(2 L^2 - 8 L R + 7 R^2)}{128R^2(R-L)L^{1/2}(2R-L)^{1/2}}
- \frac{L^{1/2} (10 L^2 - 40 L R + 27 R^2)(R_0''+R_1'')}{384R(R-L)(2R-L)^{3/2}} \nonumber \\
&
+ \frac{\scriptstyle 
L^{3/2}(24 L^4 - 192 L^3 R + 456 L^2 R^2 - 384 L R^3 + 79 R^4)((R_0'')^2+(R_1'')^2)
- 2L^2 R^2 (32 L^2 - 64 L R + 17 R^2)R_0'' R_1''}{1536(R-L)(2R-L)^{5/2}(2(R-L)^2 -R^2)} \nonumber \\
&-\frac{L^{3/2} R(R_0^{(4)}+R_1^{(4)})}{192(2R-L)^{3/2}}. \label{tau2.asym.R}
\end{align}

\subsection{Ellipse-hyperbola lens billiards}
Let $a, b, p, q$ be four positive numbers with $a>p$,  $Q(a,b,p,q)$ be the domain
bounded by the ellipse $\frac{x^2}{a^2} + \frac{y^2}{b^2}=1$
and the right branch of the hyperbola $\frac{x^2}{p^2} - \frac{y^2}{q^2}=1$.
See Fig.~\ref{ell.hyp.table}. 
The domain  $Q(a,b,p,q)$ resembles an ellipse-hyperbola lens, a special type of aspheric lenses.
We start with a special case that $a^2-b^2=p^2+q^2=:c^2$.
In this case, the two components of the boundary $\pa Q(a,b,p,q)$ are confocal conic curves
and the corresponding ellipse-hyperbola lens billiard is completely integrable.  
Consider the periodic orbit $\cO_2$ of period $2$ of the billiards  along the $x$-axis. 
Then the distance is $L=a-p$, 
the radius of the ellipse at $(a,0)$ is $R_0=\frac{b^2}{a}=a-\frac{c^2}{a}$,
the radius of the hyperbola at $(p,0)$ is $R_1=\frac{q^2}{p}=\frac{c^2}{p}-p$. 
As $p$ increases from $0$ to $c$,
$L$ decreases from $a$ to $a-c$, while $R_1$ decreases from $\infty$ to $0$. 
Moreover, $R_1=L=R_0$ when $p=p_c:=\frac{c^2}{a}$,
$R_1 >L >R_0$ when $p<p_c$ and
$R_1 <L <R_0$ when $p>p_c$. It follows that the periodic orbit $\cO_2$ is hyperbolic
when $p\in (0,c)\backslash \{p_c\}$ and is parabolic when $p=p_c$.
Next we will assume $c=1$ and
consider two families of perturbations of the confocal  ellipse-hyperbola lens billiards:
(1) shift the confocal hyperbola with $p=p_c=\frac{1}{a}$; 
(2) deform the confocal hyperbola with $p\neq p_c=\frac{1}{a}$.

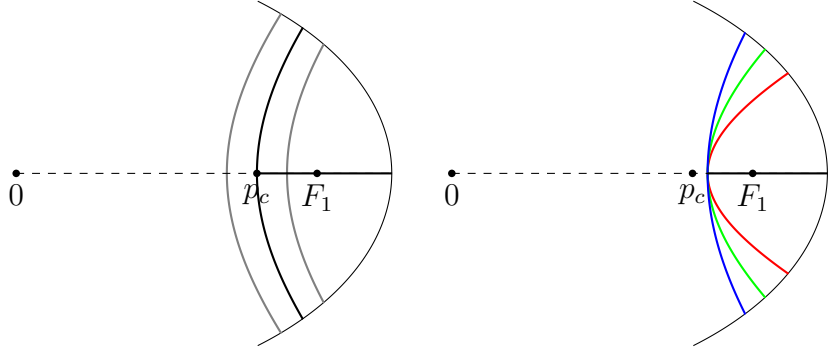
\begin{figure}[htbp]
\tikzmath{
\a=5;
\b=3;
\c=4;
\p=3.2;
\q=3;
\s=0.4;
\m=3.4;
\n=1.56205;
\r=2.33238;
}
\begin{tikzpicture}
\fill (0,0) node[below]{0} circle[radius=0.05];
\fill (\c,0) node[below]{$F_1$} circle[radius=0.05];
\clip[domain=-90:90, samples=360] plot ({\a*cos(\x)}, {\b*sin(\x)});
\draw[thick, domain=-50:50, samples=100] plot ({\a*cos(\x)}, {\b*sin(\x)});
\draw[domain=-0.8:0.8, samples=100, thick] plot ({\p*cosh(\x)}, {\q*sinh(\x)});
\draw[gray, thick, domain=-0.8:0.8, samples=100] plot ({\s+\p*cosh(\x)}, {\q*sinh(\x)});
\draw[gray, thick, domain=-0.8:0.8, samples=100] plot ({-\s+\p*cosh(\x)}, {\q*sinh(\x)});
\draw[dashed] (0,0) -- (\a,0);
\coordinate (pc) at (\p,0);
\fill (pc) circle[radius=0.05]  node[below]{$p_c$};
\draw[thick] (pc) -- (\a,0);
\end{tikzpicture}
\quad
\begin{tikzpicture}
\fill (0,0) node[below]{0} circle[radius=0.05];
\fill (\c,0) node[below]{$F_1$} circle[radius=0.05];
\clip[domain=-90:90, samples=360] plot ({\a*cos(\x)}, {\b*sin(\x)});
\draw[thick, domain=-50:50, samples=100] plot ({\a*cos(\x)}, {\b*sin(\x)});
\draw[red, thick, domain=-0.8:0.8, samples=100] plot ({\m*cosh(\x)}, {\n*sinh(\x)});
\draw[green, thick, domain=-0.8:0.8, samples=100] plot ({\m*cosh(\x)}, {\r*sinh(\x)});
\draw[blue, thick, domain=-0.8:0.8, samples=100] plot ({\m*cosh(\x)}, {(1+\r)*sinh(\x)});
\draw[dashed] (0,0) -- (\a,0);
\draw[thick] (\m,0) -- (\a,0);
\coordinate (pc) at (\p,0);
\fill (pc) circle[radius=0.05]  node[below]{$p_c$};
\end{tikzpicture}
\caption{Some ellipse-hyperbola lenses with fixed ellipse.
(1) the confocal hyperbola with $p=p_c$ (black) and two shifts (gray);
(2) a confocal hyperbola with $p>p_c$ (red) and two deformations
$q=q_2$ (green) and $q>q_2$ (blue).}\label{ell.hyp.table}
\end{figure}

\begin{example}\label{EH.shift}
We start with the confocal ellipse-hyperbola lens $Q(a, b, p, q)$ that $b^2= a^2 -1$,
$p=p_c=\frac{1}{a}$, $q^2= 1- p^2 = 1 - a^{-2}$ and shift the hyperbola horizontally by $s$. 
It follows that  the corresponding periodic orbit $\cO_2$ is elliptic
for all $0<|s|< a- \frac{1}{a}$.
In the following we will assume that
the elliptic periodic orbit $\cO_2$  satisfies  the nonresonance assumption (B1)$'$  that
$|s|\in (0, a- \frac{1}{a})\backslash \{\frac{1}{\sqrt{2}(a- \frac{1}{a})}\}$.
Plugging $L=a-\frac{1}{a} -s$, 
$R_0=R_1=a- \frac{1}{a}$,
$R_0''=\frac{3a}{b^2}-\frac{3}{a}=\frac{3}{a (a^2-1)}$ 
and $R_1''=\frac{3p}{q^2}+\frac{3}{p}=\frac{3a^3}{a^2- 1}$ into  \eqref{tau1.asym.R}, 
we have
\begin{align}
\tau_1(F^2, P)
&=\frac{a}{4 (a^2 - 1)}
-\frac{3 (a^4 +1) (a^2 - a s -1)}{ 8 a (a^2 -1) (a^2 + a s - 1)}. \label{EH.conf.tau1}
\end{align}
Note that the only zero of the twist coefficient $\tau_1$ is when 
$s=\frac{(a^2-1)(3 a^4 - 2 a^2 + 3)}{a (3 a^4 + 2 a^2 + 3)}$. 
Set $\rho_1(a):=\frac{3 a^4 - 2 a^2 + 3}{ 3 a^4 + 2 a^2 + 3} \in (\frac{1}{2},1)$.
It follows that $\tau_1 \neq 0$ and the periodic orbit $\cO_2$ is nonlinearly stable 
as long as $|s|\in (0, a- \frac{1}{a})\backslash \{\frac{1}{\sqrt{2}(a- \frac{1}{a})}\}$
and  $s\neq \rho_1(a) (a-\frac{1}{a})$. 
\end{example}

\begin{remark}
Consider the resonance case  (B1)$'$   
that $s=s_{\pm}:= \pm \frac{1}{\sqrt{2}}(a-\frac{1}{a})$.
We have  
\begin{align}
c_{03}(s_{+})=-\frac{\bar{\lambda}}{8}\frac{\sqrt{4-2 \sqrt{2}} \left(a^2+1\right)}{\left(\sqrt{2}+2\right)^{3/2} a} \neq 0, \quad
c_{03}(s_{-})=\frac{\bar{\lambda}}{8}\frac{\sqrt{4 \sqrt{2}+6} \left(a^2+1\right)}{\left(2-\sqrt{2} \right) a} \neq 0.
\end{align} 
Therefore,  there is no  first-step Birkhoff transformation for the billiard map at the periodic orbit $\cO_2$ at the resonance $s=s_{\pm}$. 
It follows from Proposition~\ref{pro.equi}
that the billiard map is not locally analytically integrable around the periodic orbit $\cO_2$.
\end{remark}

\noindent\textbf{Example \ref{EH.shift} (continued).}
We consider the case that $\tau_1=0$. This happens exactly when 
$s=\rho_1(a) (a- \frac{1}{a})$, or equally,
$L=a- \frac{1}{a}- s =(a- \frac{1}{a})\frac{4a^2}{ 3 a^4 + 2 a^2 + 3}$.
The nonresonance assumptions (B1)$'$ and (B2)$'$ 
are satisfied when $\rho_1(a) \notin \{\frac{1}{\sqrt{2}},\frac{\sqrt{3}}{2}\}$. 
Plugging $R_0^{(4)}=-\frac{3\left(4a^2 -3\right)}{a(a^2 -1)^3}$ and 
$R_1^{(4)}=-\frac{3 (p^2 - 3 q^2)}{p q^6}=-\frac{3a^5 (4 - 3a^2)}{(a^2 -1)^3}$
into \eqref{tau2.asym.R}, 
we have
\begin{align}
\tau_2
&=\frac{\scriptstyle a (729 a^{20}-2214 a^{18}-315 a^{16}-2568 a^{14}-2206 a^{12}-1188 a^{10}-2206 a^8-2568 a^6-315 a^4-2214 a^2+729)}{96 \sqrt{6} (a^2-1)^2 (a^4+1)^{3/2} (3 a^4-2 a^2+3) (9 a^8-36 a^6+22 a^4-36 a^2+9)}.
\end{align}
Note that $\tau_2$ has exact one zero at $a=a_{\ast}\approx 1.87861$.
It follows that if $a\neq a_{\ast}$ and  
$\rho_1(a) \notin \{\frac{1}{\sqrt{2}},\frac{\sqrt{3}}{2}\}$, 
then the orbit $\cO_2$ is nonlinearly stable.

\begin{example}
Consider the second family of the hyperbola-ellipse lenses by changing the parameter $q$
while fixing some $p\neq p_c=\frac{1}{a}$. We divide our discussion into two subcases
according to whether $p>\frac{1}{a}$ or $p< \frac{1}{a}$. 

(1) $p\in (\frac{1}{a}, 1)$: let $q$ increase from 
the confocal parameter $q_1=\sqrt{1 - p^2}$.
Note that both $L=a-p$ and $R_0= a- \frac{1}{a}$ are fixed in this process, and $L< R_0$.
Let $q_2:=\sqrt{p (a-p)}$ be the parameter value such that the radius 
$R_1=\frac{q^2}{p}=L$.
Then the periodic orbit $\cO_2$ remains hyperbolic for $q\in (q_1, q_2)$,
becomes parabolic at $q=q_2$ and then stays elliptic for $q> q_2$.
Next we assume $q>q_2$. 
Since $R_0=a- \frac{1}{a}<  2L=2(a-p)$, we have 
\begin{align*}
(1-\frac{L}{R_0})(1-\frac{L}{R_1}) < 1-\frac{L}{R_0} < \frac{1}{2}.
\end{align*}
Therefore  the resonance $\lambda^4=1$ does not occur in this family.
It follows from \eqref{tau1.asym.F} that
the first twist coefficient of the orbit $\cO_2$ is given by
\begin{align}
\tau_1(F^2, P)
=\frac{1}{8} \Big(\frac{a^2 p+a q^2 - p}{q^2 (a^2- 1)}
+\frac{ap(a-p) \Big(\frac{3(p^2+q^2)(a p-1)}{a q^2 (a p-p^2-q^2)}
+\frac{3 (a p-p^2-q^2)}{p (a^2-1)(a p -1)}\Big)}{a (p^2+ q^2)-  p}\Big). \label{tau1.ell.hyp}
\end{align}
Note that for $q$ sufficiently large, 
$\tau_1(F^2, P) \approx\frac{1}{8}(\frac{a}{(a^2- 1)}-\frac{3(a-p)}{ (a^2-1) (a p -1)})$, 
which vanishes at $p=\frac{4a}{a^2 +3} \in (\frac{1}{a}, 1)$.
It follows that there are cases for which $\tau_1(F^2, P)$ vanishes.

(2) $p\in (0, \frac{1}{a})$: let $q$ decrease from 
the confocal parameter $q_1=\sqrt{1 - p^2}$.
Note that $L=a-p$ and $R_0=a- \frac{1}{a}$ are fixed, and $L> R_0$.
Let $q_2:=\sqrt{p(a-p)}$ be the parameter value such that the radius 
$R_1=\frac{q^2}{p}=L$,
and $q_3:=\sqrt{p(\frac{1}{a} -p)}$ be the parameter value such that the radius 
$R_1=\frac{q^2}{p}=L-R_0$.
Then the periodic orbit $\cO_2$ is hyperbolic for $q\in (0,q_3)\cup (q_2, q_1)$,
is parabolic at $q \in \{q_2, q_3\}$ and is elliptic for $q\in (q_3, q_2)$.
The nonresonance assumption (B1)
$(\frac{L}{R_0}-1)(\frac{L}{R_1}-1) \neq \frac{1}{2}$ holds 
exactly when $R_1(q) \neq\frac{2L(L- R_0)}{2L- R_0}$. 
Note that such a value always lies in the interval $(R_1(q_3), R_1(q_2))=(L-R_0, L)$, 
and hence the resonance $\lambda^4=1$ does  occur.
The corresponding value for $q$ at resonance is 
$q^2=\frac{2p(a-p)(1-a p)}{a^2-2 a p+1}$. 
If $q^2 \neq \frac{2p(a-p)(1-a p)}{a^2-2 a p+1}$,
then the first twist coefficient of the orbit $\cO_2$ is given by \eqref{tau1.ell.hyp}.
Note that $\tau_{1}(F^2, P) \to -\infty$ at both limits $q\searrow q_3$ and $q\nearrow q_2$.
\end{example}

\begin{remark}
We consider the resonance case that $q^2=\frac{2p(a-p)(1-a p)}{a^2-2 a p+1}$ for
$p\in (0, \frac{1}{a})$, and obtain that
\begin{align*}
c_{03}=-\frac{\bar{\lambda}}{8}
\frac{ a (-a^2 -1 + 2 a^3 p + 2 a p - 3 a^2 p^2 + p^2) ( (a - p)^2 + (1- a p)^2)}{p (a^2-1)^2 (a-p) (1-a p)^2}.
\end{align*}
Then $c_{03}=0$ if and only if 
$p=p_{\ast}(a):=\frac{a^3+a  -(a^2-1)\sqrt{a^2+1}}{3 a^2-1} \in (0, \frac{1}{a})$.
Therefore, assuming the periodic orbit $\cO_2$ is at resonance with $\lambda^4=1$, 
or equally, 
$q^2=\frac{2p(a-p)(1-a p)}{a^2-2 a p+1}$, we have
\begin{enumerate}
\item if $p \neq p_{\ast}(a)$, then there is no  first-step Birkhoff transformation, and the billiard map is not locally analytically integrable around the periodic orbit  $\cO_2$;

\item if $p=p_{\ast}(a)$, 
then $c_{03}=0$ and there exists a first-step Birkhoff transformation for the billiard map at the periodic orbit $\cO_2$. Moreover,  
\begin{align}
\tau_1=
-\frac{8 a + 54 a^3 + 50 a^5 + 2 a^7 - 2 a^9 +(1+ 27 a^2+ 53 a^4 -  3 a^6+ 2 a^8)\sqrt{a^2+1}}
{4 (a^2-1) (\sqrt{a^2+1}+2 a) (\sqrt{a^2+1} a -a^2+1)^2 (a^3+a -\sqrt{a^2+1}(a^2-1))}.
\end{align}
One can check that $\tau_1<0$ and hence  the periodic orbit $\cO_2$ is nonlinearly stable.
\end{enumerate}
\end{remark}

The rest of the paper is organized as follows.
In \S \ref{sec.taylor} we provide some preliminary results on the billiard map
and derive the Taylor expansion of the billiard map $F$ at the point $P$ for the symmetric billiards. 
In \S \ref{sec.tau1} we obtain
the Birkhoff normal form with the first twist coefficient $\tau_1(F,P)$ for the symmetric billiards.
The explicit formula about $\tau_1(F,P)$ is not new. What is new is the detailed justification
about the Birkhoff transformation and Birkhoff normal form.
In \S \ref{sec.tau2} we obtain an explicit formula of the second twist coefficient $\tau_2(F,P)$
in terms of the geometric quantities of the billiard table for the symmetric billiards. 
A comparison between our approximation of rotation numbers
with Ko{\l}odziej's formula is given in \S \ref{sec.Kol}.
In \S \ref{asym.table} we consider the asymmetric billiards and obtain an explicit formula
for the second twist coefficient $\tau_2(F^2,P)$ utilizing the dependence of $\tau_2(F^2,P)$
on the higher order derivatives of the radii of curvature of the boundary of the table. 
There are several occasions that we need to solve a system of linear equations with a large
number of unknowns. Direct approaches do not work well with badly conditioned coefficient matrices since a small change of the input could lead to a large change of the outcome. 
Instead, we manipulate the functions and  divide the unknowns into several groups of unknowns of much smaller sizes. Finally, we  check the formula for $\tau_2(F^2,P)$ is indeed 
correct.

\section{Preliminaries}
\label{sec.taylor}

Let $a(t)=\sum_{n\ge 1} a_{2n}t^{2n}$ and $b(t)=\sum_{n\ge 1} b_{2n}t^{2n}$, 
$t\in(-\eps, \eps)$, be two locally defined even functions, 
$Q(L, a, b)$ be the domain bounded by 
the curves $\gamma_0(t)=(a(t), t)$ and $\gamma_1(t)=(L-b(t), -t)$, $t\in(-\eps, \eps)$,
with two horizontal segments connecting the corresponding endpoints.
See Fig.~\ref{btable}. We will call $Q(L, a, b)$ a billiard table.
Let $M$ be the space of unit vectors $v\in T_{\pa Q(L, a, b)}\bR^2$
pointing inside of the table, which is called the phase space of the billiards.
Let $F$ be the billiard map sending a unit vector to the new unit vector after reflection
on the boundary $\pa Q(L,a, b)$. See \cite{ChMa} for more details.

Instead of using a global parameter for the whole boundary $\pa Q(L,a ,b)$,
we choose to use one local parameter for each of the two curves $\gamma_j$, $j=0, 1$.
That is, let $s$ be the arc-length parameter of $\gamma_j$ with $s(\gamma_j(0))=0$, $j=0, 1$.
Let $\theta$ be the angle from the positive tangent vector  $\dot\gamma_j(s)$  (counterclockwise orientation)
to the outgoing billiard trajectory starting at $\gamma_j(s)$, and $u=-\cos\theta$. 
Then a point in the phase space $M$ has a well-defined coordinate $(s,u)$.
For example, the points $P$ and $F(P)$ of the periodic orbit $\cO_2$
are given by the coordinates $P=(0, 0)$ and $F(P)=(0, 0)$, respectively.

For every small positive number $\delta$,
the open set $U=(-\delta, \delta) \times (-\delta, \delta)$ 
in the coordinate system $(\bR^2,(s, u))$ determines  an open neighborhood $U_{P}$ of $P$
in the phase space of the billiards.
Note that $U=(-\delta, \delta) \times (-\delta, \delta)$
also determines an open neighborhood $U_{F(P)}$ of $F(P)$ in the phase space of the billiards.
Generally speaking, the maps $F: U_{P} \to U_{F(P)}$ and $F: U_{F(P)} \to U_{P}$
are two different functions on the same open set $U$.

For the remaining of this section and both \S\ref{sec.tau1} and \S\ref{sec.tau2}, 
we will consider the case that $b(t)=a(t)$ and leave the general case to \S\ref{asym.table}.
Let $\text{Rot}_{\pi}$ be the rotation of the billiard table $Q(L,a,a)$ around its middle point $(\frac{L}{2}, 0)$ by $\pi$. 
Note that the table is invariant under the rotation  $\text{Rot}_{\pi}$, and the rotation $\text{Rot}_{\pi}$ commutes with the billiard map $F$.  

\begin{remark}\label{rem.identify}
The rotation $\text{Rot}_{\pi}$ identifies the phase space on $\gamma_1$ with the phase space on $\gamma_0$.  It follows that the composition $F^2$ is equal to  $(\text{Rot}_{\pi} \circ F)^2$ on a small neighborhood $U$ of the point $P=(0, 0)$.
Moreover, the map $\text{Rot}_{\pi} \circ F$ can be viewed as a diffeomorphism from the neighborhood $U$
onto its image. In the following we use (abuse) the notation $F$ for $\text{Rot}_{\pi} \circ F$ after the identification of the phase spaces on $\gamma_0$ and $\gamma_1$ via the rotation $\text{Rot}_{\pi}$.
It follows from the construction of Birkhoff Normal Form, or more precisely, from  Eq.~\eqref{formal} and \eqref{BNF}, that the twist coefficients $\tau_k(F^2, P)$, $k\ge 1$,
of the iterated billiard map $F^2$ at $P\in U_{P}$
are just twice of the corresponding twist coefficients $\tau_k(F, P)$, 
$k\ge 1$ of the one-step function $F$ at $P=(0,0)\in U$. 
\end{remark}

\subsection{Taylor polynomial of the function at the point $P$}
Let $R(s)$ be the radius of curvature of the curve $\gamma_0(s)$. 
Note that $R(s)$ is an even function of $s$ since the function $a(t)$ is an even function.
Given a point $(s,u)\in U$, let $(s_1, u_1)=F(s,u)$ be the image under $F$,
$\ell(s,s_1)=|\gamma_1(s_1)-\gamma_0(s)|$ be the length of the trajectory between the two impact points,
and $\ell(s,u)=\ell(s,s_1(s,u))$ be the corresponding function that depends on the initial point $(s,u)$.
Then the first-order partial derivatives of $s_1$ and $u_1$ with respect to $s$ and $u$  are:
\begin{align}
\frac{\pa s_1}{\pa s}&= \frac{\ell(s,u)}{R(s)\sqrt{1-u_1^2}} -\frac{\sqrt{1-u^2}}{\sqrt{1-u_1^2}};\\
\frac{\pa s_1}{\pa u}&= \frac{\ell(s,u)}{\sqrt{1-u^2}\sqrt{1-u_1^2}};\\
\frac{\pa u_1}{\pa s}&= \frac{\ell(s,u)}{R(s)R(s_1)} -\frac{\sqrt{1-u_1^2}}{R(s)} - \frac{\sqrt{1-u^2}}{R(s_1)} ;\\
\frac{\pa u_1}{\pa u}&= \frac{\ell(s,u)}{R(s_1)\sqrt{1-u^2}} -\frac{\sqrt{1-u_1^2}}{\sqrt{1-u^2}};
\end{align}
See \cite[Section 2.11]{ChMa} for more details.
See also \cite[Eq.~(1-1)]{KP05}. The length function $\ell(s,s_1)$ is a generating function of the billiard map, see \cite[Chapter 3]{Tab05}. More precisely, it satisfies
\begin{align*}
d\ell=-\cos\theta\ ds + \cos\theta_1\ ds_1 = u\ ds - u_1\ ds_1.
\end{align*}
Since $d(d\ell)=0$, it follows that $du_1\wedge ds_1= du\wedge ds$. That is, the billiard map 
$F$ preserves the $2$-form $du\wedge ds$ and is a symplectic embedding on a neighborhood $U$ of point $P=(0,0)$.
Applying the chain rule to the function $\ell(s,u)=\ell(s,s_1(s,u))$, we have
\begin{align}
\frac{\pa \ell}{\pa s}&= u - u_1\frac{\pa s_1}{\pa s}=u -u_1\Big(\frac{\ell(s,u)}{R(s)\sqrt{1-u_1^2}} -\frac{\sqrt{1-u^2}}{\sqrt{1-u_1^2}}\Big);\\
\frac{\pa \ell}{\pa u}&= -u_1\frac{\pa s_1}{\pa u}= -u_1\frac{\ell(s,u)}{\sqrt{1-u^2}\sqrt{1-u_1^2}}.
\end{align}

It follows from the expressions of the first-order partial derivatives that the higher-order partial derivatives 
$\frac{\pa^{j+k}s_1}{\pa s^j \pa u^k}$ 
and $\frac{\pa^{j+k}u_1}{\pa s^j \pa u^k}$ can be expressed recursively  using $s$, $u$,
$s_1$, $u_1$, $\ell(s,u)$, $R^{(n)}(s)$, $0\le n \le j+k-1$.
In particular, evaluating at $P=(0,0)$, we have, $(s_1, u_1)=(0,0)$, 
$\ell(0,0)=L$, $R^{(k)}(0)=R^{(k)}$, and the higher-order derivatives
$\frac{\pa^{j+k}s_1}{\pa s^j \pa u^k}(0,0)$ 
and $\frac{\pa^{j+k}u_1}{\pa s^j \pa u^k}(0,0)$
are rational functions of  $R$
and are polynomials of $L$ and $R^{(n)}$, $1\le n \le j+k -1$.
Since $R(s)$ is an even function,
a direct calculation shows that all the 2nd- and 4th-order derivatives of $s_1$ and $u_1$
with respect to $s$ and $u$ vanish at $(s,u)=(0,0)$,
the 3rd-order derivatives of $s_1$ and $u_1$ are linear with respect to $R''$,
and the 5rd-order derivatives of $s_1$ and $u_1$ are linear with respect to $R^{(4)}$ 
and are quadratic with respect to $R''$. Here we only list the 3rd-order partial derivatives\footnote{
The Mathematica codes for the calculations used in this article can be found in \cite{JZ21}.}:
\begin{align}
\frac{\pa^{3}s_1}{\pa s^3}(0,0)
&=\frac{L^3 - 6 L^2 R + 11L R^2 - 6R^3}{R^5} -\frac{L}{R^2}R'';
\\
\frac{\pa^{3}s_1}{\pa s^2 \pa u}(0,0)
&=\frac{L^3 - 5 L^2 R + 7 L R^2 - 2 R^3}{R^4};
\\
\frac{\pa^{3}s_1}{\pa s \pa u^2}(0,0)& =\frac{L (L - 2 R)^2}{R^3};
\\
\frac{\pa^{3}s_1}{ \pa u^3}(0,0) &=\frac{L (L^2 - 3 L R + 3 R^2)}{R^2};
\end{align}
\begin{align}
\frac{\pa^{3}u_1}{\pa s^3}(0,0) 
&=\frac{2 R - L }{R^4} + \frac{-L^3 + 3 L^2 R - 4L R^2 + 2 R^3}{R^5}R'';
\\
\frac{\pa^{3}u_1}{\pa s^2 \pa u}(0,0) 
&=\frac{2 R - L  }{R^3} +\frac{ - L^3 + 2 L^2 R - L R^2}{R^4}R''; 
\\
\frac{\pa^{3}u_1}{\pa s \pa u^2}(0,0) &=\frac{2 R-L}{R^2} + \frac{ - L^3+ L^2 R}{R^3}R'';
\\
\frac{\pa^{3}u_1}{ \pa u^3}(0,0) &=-\frac{L^3}{R^2}R''.
\end{align}
The 5th-order Taylor polynomial of the function $F(s,u)=(s_1, u_1)$ 
around $P=(0, 0)$ can be written as
\begin{align}
s_1 & = \ta_{10}s + \ta_{01} u 
+\sum_{j+k=3} \ta_{jk}s^j u^k+ \sum_{j+k=5} \ta_{jk}s^j u^k+h.o.t., \label{s1su}
\\
u_1 & = \tb_{10}s + \tb_{01} u 
+\sum_{j+k=3} \tb_{jk}s^j u^k+ \sum_{j+k=5} \tb_{jk}s^j u^k+h.o.t., \label{u1su}
\end{align}
where $\ta_{jk}=\frac{1}{j!k!}\frac{\pa^{j+k}s_1}{\pa s^j \pa u^k}(0,0)$
 $\tb_{jk}=\frac{1}{j!k!}\frac{\pa^{j+k} u_1}{\pa s^j \pa u^k}(0,0)$,
and h.o.t. stands for higher order terms.
Note that $\ta_{10}=\tb_{01}$, $\deg(\ta_{jk})=1-j$ and $\deg(\tb_{jk})=-j$, 
and all the terms in \eqref{s1su} and \eqref{u1su} are homogeneous
with respect to scaling.

\section{Birkhoff Normal Form with First Twist Coefficient}
\label{sec.tau1}

In this section we will find a symplectic coordinate transformation under which the billiard map
$F$ is of Birkhoff Normal Form of order $3$ at the point $P$.
The main reference for this section is the classic book \cite{SiMo}, see also \cite{Moe90}.
We start with a linear coordinate transformation under which $D_{P}F$ acts as a rigid rotation.
Given the tangent matrix $D_{P}F$ in \eqref{DF},
let $\eta=(-\ta_{01}/\tb_{10})^{1/4}$ be a positive number, and 
$v=\begin{bmatrix} i\eta \\ \eta^{-1} \end{bmatrix}$ be a complex vector. 
Then $v$ is an eigenvector of the matrix $D_{P}F$ corresponding to 
the eigenvalue $\lambda = \ta_{10} - i \sqrt{- \ta_{01} \tb_{10}}$.
Let  $\theta\in (0,2\pi)$ be an argument of the eigenvalue $\lambda$.
Then we introduce a new coordinate system 
$(x,y)=(s/\eta, \eta u)$ and set $(x_1,y_1)=(s_1/\eta, \eta u_1)=F(x,y)$.
It follows from  \eqref{s1su} and \eqref{u1su} that
\begin{align}
\begin{bmatrix} x_1 \\ y_1 \end{bmatrix}
=R_{\theta}\begin{bmatrix} x \\ y \end{bmatrix}
+
\begin{bmatrix} \sum_{j+k=3} a_{jk} x^j y^k+ \sum_{j+k=5} a_{jk} x^j y^k \\ 
\sum_{j+k=3} b_{jk} x^j y^k+ \sum_{j+k=5} b_{jk} x^j y^k\end{bmatrix}+h.o.t. ,
\end{align}
where $R_{\theta}=\begin{bmatrix}\cos\theta & -\sin\theta \\ \sin\theta & \cos\theta \end{bmatrix}$
is a rotation matrix, 
and the coefficients $a_{jk}$ and $b_{jk}$  are given by
\begin{align} 
a_{jk}=\eta^{j-k-1}\ta_{jk}, 
\quad
b_{jk}=\eta^{j-k+1}\tb_{jk},
\end{align}
where $j+k =1, 3, 5$, respectively. For convenience, we write down the linear terms:
\begin{align}
a_{10}= b_{01}=\frac{1}{R}(L -R), \quad 
a_{01} = - b_{10} = \frac{1}{R}\sqrt{L(2R-L)}.
\end{align}
Note that $\deg(\eta)=1/2$. So $\deg(x)=\deg(y)=1/2$,
and $\deg(a_{jk})=\deg(b_{jk})=\frac{1-j-k}{2}$.
In particular, $\deg(a_{jk})=\deg(b_{jk})=0, -1, -2$ when $j+k=1, 3, 5$, respectively.
Observe that $\lambda = a_{10} - i \sqrt{- a_{01} b_{10}} = a_{10} - i a_{01}$ also holds.

\subsection{Introduce complex coordinates}\label{tau1.complex}
One can embed $\bR^2\subset \bC^2$
and view the one-step function $F: (x,y)\mapsto (x_1, y_1)$ 
as a complex symplectic function on a small neighborhood 
of $(0,0)\in\bC^2$ with the same real coefficients. 
Note that this is only a formal extension,
and the true dynamics happens in the 2D real subspace  $\bR^2\subset \bC^2$.
To compute the twist coefficients, 
we consider the following complex coordinate system $(z,w)$ on $\bC^2$:
\begin{align}
z=x +i y; \quad w= x-iy.
\end{align}
It follows from this definition that the real subspace $\bR^2=\{(x,y)\in \bC^2: x,y \in \bR\}$
is transformed to $\{(z,w)\in \bC^2: w=\bar z\}$, which can be also be viewed as a real subspace.
Let $(z_1, w_1)=F(z,w)=(x_1 + iy_1, x_1 -iy_1)$. 
Here we use $F$ (by some abuse of notation) for the new function representing
the transformation in the  $(z,w)$-coordinate. 
We also use this convention for later coordinate transformations.
Then  $w_1=\bar z_1$ if and only if $w=\bar z$ (since both are real 2D subspaces
and the map $F$ is a local diffeomorphism). We will use this characterization 
repeatedly in the remaining of this paper.

Note that $z$ and $w$ span the eigenspaces of the tangent matrix $D_P F$, and the Taylor polynomial of the function $(z_1, w_1)=F(z,w)$ has a diagonalized
linear term and vanishing second- and fourth-order terms:
\begin{align}
z_1&=\lambda(z +\sum_{j+k=3} c_{jk}z^j w^k +\sum_{j+k=5} c_{jk}z^j w^k)+h.o.t., \label{z1zw}\\
w_1&=\bar\lambda(w +\sum_{j+k=3} \bar c_{jk}w^j z^k
 +\sum_{j+k=5} \bar c_{jk}w^j z^k)+h.o.t., \label{w1zw}
\end{align}
where
\begin{align}
c_{30}&=2^{-3}\bar\lambda(a_{30}+ib_{30} - ia_{21}+ b_{21}
 -a_{12} -ib_{12} + ia_{03}-b_{03}),\\
c_{21}&=2^{-3}\bar\lambda(3a_{30}+3ib_{30} -ia_{21}+ b_{21}
+a_{12} +ib_{12} - 3ia_{03} +3b_{03}), \label{c-21}\\
c_{12}&=2^{-3}\bar\lambda(3a_{30}+3ib_{30}+ia_{21}- b_{21}
+a_{12} +ib_{12} + 3ia_{03} -3b_{03}),\\
c_{03}&=2^{-3}\bar\lambda(a_{30}+ib_{30} + ia_{21}- b_{21} 
-a_{12} -ib_{12} - ia_{03} +b_{03}). \label{c-03}
\end{align}
Note that $\deg(c_{jk}) =-1$ for all $j+k=3$.
The fifth-order terms $c_{jk}$, $j+k=5$, are given in a similar way.
Moreover, we have $\frac{\pa (z_1, w_1)}{\pa(z, w)}=1$  since $F$ is symplectic
and the transformation $(x,y) \to (z,w)$ has constant Jacobian.
Plugging \eqref{z1zw} and \eqref{w1zw} into the equation
$\frac{\pa (z_1, w_1)}{\pa(z, w)}=1$ and comparing
the coefficients of the second-order terms, we have
\begin{align}
c_{12} =-3\bar c_{30},
\quad
\bar c_{21} = -c_{21}. \label{jacobi3}
\end{align}
In particular, $c_{21}$ is purely imaginary, say $c_{21}=i\tau_1$ 
for some real number $\tau_1$.
As we will see later, $\tau_1$ is the first twist coefficient of the function $F$ at the point $P$.
We can simplify and get an explicit formula for the quantity $\tau_1$ in terms of 
the geometric characteristics of the billiard table:
\begin{align}
\tau_1(F,P)= \frac{1}{8R} - \frac{L}{8(2R-L)}R''. 
\end{align}
This completes the derivation of the formula \eqref{tau1.sym.F}.
It also confirms that the quantity $\tau_1$ is homogeneous and $\deg(\tau_1)=-1$.

For later convenience, we write down an explicit formula for the term $c_{03}$ in the resonance
case $\lambda^4=1$, or equally, $L=R$. In this case, $\lambda = -i$, and 
\begin{align}
c_{03}=\frac{i}{8}\cdot (-\frac{3+ R R''}{3 R})
=-\frac{i}{24 R}(3+ R R''). \label{c03LR}
\end{align}
It follows that $c_{03}=0$ if and only if $R''= - \frac{3}{R}$.

\subsection{An intermediate transformation}\label{tau1.intermediate}
In this subsection we make use of the assumption (A1) that $\lambda^4 \neq 1$, or equally,
$\frac{L}{R}\in (0, 2)\backslash \{1\}$.
We want to find a transformation $(z,w)\to (z', w')$ under which most of the third-order terms
$c_{jk}$, $j+k=3$, vanish.
To this end, we consider an intermediate coordinate transformation 
with undetermined complex coefficients $d_{jk}$, $j+k=3$:
\begin{align}
z'&=z+p_3(z,w) = z +\sum_{j+k=3} d_{jk}z^j w^k, \label{tran.inter.z}\\
w'&=w+\bar p_3(w,z)= w +\sum_{j+k=3} \bar d_{jk}w^j z^k. \label{tran.inter.w}
\end{align}
This transformation preserves the real subspace $\{(z,w)\in \bC^2: \bar z=w\}$,
but is not necessarily symplectic.
The function $F(z', w')=(z_1', w_1')$ satisfies $\overline{z_1'}=w_1'$, and 
\begin{align}
z_1' &=z_1 +\sum_{j+k=3} d_{jk}z_1^j w_1^k \nonumber \\
&=\lambda(z +\sum_{j+k=3} c_{jk}z^j w^k) 
+ \sum_{j+k=3} d_{jk}\lambda^{j-k} z^j w^k + h.o.t.  \nonumber \\
&=\lambda \big(z'+\sum_{j+k=3} \big(-d_{jk} +c_{jk}+ d_{jk}\lambda^{j-k-1}\big)(z')^j (w')^k \big)
+h.o.t. \label{norm1}
\end{align}
Note that the coefficient of $(z')^2 w'$ in the above equation 
does not change regardless of the choice of $d_{21}$. 
Without loss of generality, we set $d_{21}=0$.
Equating $-d_{jk} +c_{jk}+d_{jk}\lambda^{j-k-1}=0$ 
for each $(j,k) \in \{(3,0), (1,2), (0, 3)\}$, we have
\begin{align}
d_{30}=\frac{c_{30}}{1-\lambda^2}, \quad
d_{12}=\frac{c_{12}}{1- \overline{\lambda}^2}=-3\bar d_{30}, \quad
d_{03}=\frac{c_{03}}{1- \overline{\lambda}^4}, \label{d-jk}
\end{align}
since $c_{12}=-3\bar c_{30}$,  see \eqref{jacobi3}. Note that $\deg(d_{jk})=-1$.

\begin{remark} \label{loc.ana.int}
When  $c_{03}\neq 0$,  the assumption (A1) $\lambda^4 \neq 1$ is needed 
when defining $d_{03}$ in \eqref{d-jk}. 
If $c_{03}=0$, then the transformation $(z,w) \to (z', w')$ can be defined
even in the resonance case  $\lambda^4= 1$. 
This happens in the elliptic billiards case, see Remark \ref{ellipse.resonance}.
On the other hand, it follows from Proposition \ref{pro.equi} that, 
in the  resonance case  $\lambda^4= 1$, 
$c_{03}\neq 0$ implies that the billiard map $F$ is not  locally analytically integrable.
This happens in the lemon billiards case, see Remark \ref{re.lemon} and the discussion
after Proposition \ref{pro.equi}.
\end{remark}

Applying \eqref{d-jk} to \eqref{norm1}, 
we get $z_1'=\lambda(z' + c_{21}(z')^2 w') + h.o.t.$.
Using the properties that $c_{21} = i\tau_1$ is purely imaginary,
$\lambda=e^{i\theta}$ and restricting to the subspace $\{(z,w)\in \bC^2: \bar z=w\}$,
we have
\begin{align}
z_1'=e^{i\theta}(z' +  i\tau_1|z'|^2 z') + h.o.t.
=e^{i(\theta+ \tau_1|z'|^2)}z'+ h.o.t. \label{realtran}
\end{align}
This explains the geometric meaning of the term $\tau_1$: it measures
the amount of twisting of rotations around the elliptic point $P$.

\begin{remark}
The transformation $(z, w) \to (z', w')$ induces a coordinate transformation
$(x,y) \to (x', y')$ of the real variables $(x,y)$. 
However, this real transformation is not necessarily symplectic
and hence the function $F$ is not necessarily symplectic in the coordinate system $(x', y')$. 
\end{remark}

\subsection{A symplectic coordinate transformation}\label{tau1.symp}
To find the corresponding symplectic coordinate transformation in the $(x,y)$-plane,
we will use the generating function construction. 
First, we note that $\frac{\pa }{\pa z}p_3(z, w) = -\frac{\pa}{\pa w} \bar p_3(w,z)$.
So the polynomial 
\begin{align}
s_4(z, w) = \frac{\bar c_{03}}{\lambda^4 -1} \frac{z^4}{4}
-\frac{c_{03}}{\overline{\lambda}^4 -1} \frac{w^4}{4}
- \frac{ c_{30}}{\lambda^2 -1}z^3w + \frac{\bar c_{30}}{\overline{\lambda}^2 -1}zw^3
\end{align}
satisfies that $\frac{\pa s_4}{\pa w}=p_3(z, w)$ and $\frac{\pa s_4}{\pa z}=-\bar p_3(w, z)$.
Note that the complex polynomial $s_4$ is purely imaginary on the subspace
$\{(z,w)\in\bC^2: \bar z=w\}$.
Therefore, the function $g_4(x,y):=\frac{i}{2}s_4(x+iy, x-iy)$ 
is a real homogeneous polynomial of 
the real variables $x$ and $y$.
Let $G_4(x, Y)=xY + g_4(x, Y)$ and consider the transformation $(x,y) \to (X, Y)$,
where $X$ and $y$ are given by
\begin{align} 
X :=\frac{\pa G_4}{\pa Y} &= x+\frac{i}{2}(-\bar p_3(x-iY, x+iY)(i)+p_{3}(x+iY, x-iY)(-i))
\nonumber \\
&=x+ \Re(p_{3}(x+iY, x-iY))=: x+ \sum_{j+k=3}p_{jk}x^j Y^k, \label{gene.3X}\\
y :=\frac{\pa G_4}{\pa x} & = Y+\frac{i}{2}(-\bar p_3(x-iY, x+iY)+p_{3}(x+iY, x-iY)) 
\nonumber  \\
&=Y - \Im(p_{3}(x+iY, x-iY))=: Y- \sum_{j+k=3}q_{jk}x^j Y^k. \label{gene.3Y}
\end{align}
This defines implicitly a  symplectic coordinate  transformation $(x,y) \mapsto (X, Y)$ 
in a small neighborhood of the origin $(0,0)$  in $\bR^2$. 
It follows from \eqref{gene.3Y} that $Y$ is homogeneous and $\deg(Y)=\frac{1}{2}$,
and then from \eqref{gene.3X} that $X$ is homogeneous and $\deg(X)=\frac{1}{2}$.
Moreover, we get the explicit local expansions from \eqref{gene.3X} and  \eqref{gene.3Y}:
\begin{align}
X&=x+\sum_{j+k=3}p_{jk}x^j y^k+ h.o.t., \label{coor.tran.X}  \\
Y&=y+\sum_{j+k=3}q_{jk}x^j y^k+ h.o.t..  \label{coor.tran.Y}
\end{align}

Consider the new complex coordinate 
$(Z, W):=(X+ i Y, X-iY)$. It is easy to see that
\begin{align*}
Z&= X +i Y = x + iy + \Re(p_{3}(x+iy, x-iy)) + i\Im(p_{3}(x+iy, x-iy)) + h.o.t. \\
&=z + p_3(z, w)+ h.o.t.
\end{align*}
One can check that the image $(Z_1, W_1)$ of $(Z,W)$ under the map $F$  satisfies
\begin{align}
Z_1&=z_1+ p_3(z_1, w_1) + h.o.t.
=\lambda (Z+ i\tau_1 Z^2 W) + h.o.t.  \nonumber\\
&=e^{i\theta}(1+i\tau_1 |Z|^2) Z + h.o.t. 
=e^{i(\theta+ \tau_1 |Z|^2)}Z + h.o.t. \label{symptran}
\end{align}

\section{Birkhoff Normal Form with First Two Twist Coefficients}
\label{sec.tau2}

In this  section we will find a coordinate transformation under which the billiard map $F$
is of Birkhoff normal form with first two twist coefficients at the elliptic point $P$.
To this end, we need to find the fifth-order terms
of the coordinate transformation $(x, y) \to (X, Y)$ given by \eqref{coor.tran.X} and \eqref{coor.tran.Y}
and the Taylor expansion of the function $F$ under the coordinate system $(X, Y)$. 
To start, we assume 
\begin{align}
X&=x+\sum_{j+k=3}p_{jk}x^j y^k+\sum_{j+k=4}p_{jk}x^j y^k
+\sum_{j+k=5}p_{jk}x^j y^k+ h.o.t., \label{gene.5X} \\
Y&=y+\sum_{j+k=3}q_{jk}x^j y^k+\sum_{j+k=4}q_{jk}x^j y^k
+\sum_{j+k=5}q_{jk}x^j y^k+ h.o.t.  \label{gene.5Y}
\end{align} 
Substituting \eqref{gene.5X} and \eqref{gene.5Y} into \eqref{gene.3X} and \eqref{gene.3Y}, 
we see that the terms $p_{jk}$ and $q_{jk}$, $j+k=4$, vanish, 
and the terms $p_{jk}$ and $q_{jk}$, $j+k=5$,  are given by:
\begin{align*}
p_{50} & = p_{21}q_{30};\\
p_{41} & = p_{21}q_{21} +  2p_{12}q_{30};\\
p_{32} & = p_{21}q_{12} +  2p_{12}q_{21} +  3p_{03}q_{30} ;\\
p_{23} & = p_{21}q_{03} +  2p_{12}q_{12} +  3p_{03}q_{21} ;\\
p_{14} & =2p_{12}q_{03} +  3p_{03}q_{12} ;\\
p_{05} & =3p_{03}q_{03}.
\end{align*}

\begin{align*}
q_{50} & = q_{21}q_{30};\\
q_{41} & = q_{21}q_{21} +  2q_{12}q_{30};\\
q_{32} & = q_{21}q_{12} +  2q_{12}q_{21} +  3q_{03}q_{30} ;\\
q_{23} & = q_{21}q_{03} +  2q_{12}q_{12} +  3q_{03}q_{21} ;\\
q_{14} & =2q_{12}q_{03} +  3q_{03}q_{12} ;\\
q_{05} & =3q_{03}q_{03}.
\end{align*}
Note that $p_{jk}$ and $q_{jk}$, $j+k=5$, are homogeneous of degree $-2$.

Suppose the Taylor expansion of the function $F$ in the new coordinate system $(X,Y) \to (X_1, Y_1)$
is given by 
\begin{align}
X_1&=a_{10}X+ a_{01}Y +  \sum_{j+k=3}A_{jk} X^j Y^k 
+ \sum_{j+k=5}A_{jk}X^j Y^k + h.o.t., \label{X1XY}\\
Y_1&=b_{10}X+ b_{01}Y +  \sum_{j+k=3}B_{jk} X^j Y^k 
+ \sum_{j+k=5}B_{jk}X^j Y^k + h.o.t., \label{Y1XY}
\end{align}
with undetermined coefficients $A_{jk}$ and $B_{jk}$, $j+k=3, 5$.
To determine these coefficients, we recall
\begin{align}
X_1= & x_1 + \sum_{j+k=3}p_{jk}x_1^j y_1^k + \sum_{j+k=5}p_{jk}x_1^j y_1^k + h.o.t.  \nonumber\\
=& a_{10}x+ a_{01}y + \sum_{j+k=3}a_{jk}x^j y^k + \sum_{j+k=5}a_{jk}x^j y^k \nonumber \\
&+\sum_{j+k=3} p_{jk}\cdot \substack{
(a_{10}x+ a_{01}y +a_{30}x^3+a_{21}x^2y+a_{12}x y^2+a_{03}y^3)^j\\ 
\cdot(b_{10}x+ b_{01}y +b_{30}x^3+b_{21}x^2y+b_{12}x y^2+b_{03}y^3)^k } \nonumber \\
 &+ \sum_{j+k=5}p_{jk}(a_{10}x+ a_{01}y)^j (b_{10}x+ b_{01}y)^k + h.o.t., \label{X1xy}
\end{align}

\begin{align}
Y_1 =& y_1 + \sum_{j+k=3}q_{jk}x_1^j y_1^k + \sum_{j+k=5}q_{jk}x_1^j y_1^k + h.o.t. \nonumber \\
= & b_{10}x+ b_{01}y + \sum_{j+k=3}b_{jk}x^j y^k + \sum_{j+k=5}b_{jk}x^j y^k \nonumber \\
&+\sum_{j+k=3}q_{jk}\cdot \substack{
(a_{10}x+ a_{01}y +a_{30}x^3+a_{21}x^2y+a_{12}x y^2+a_{03}y^3)^j \\ 
\cdot(b_{10}x+ b_{01}y +b_{30}x^3+b_{21}x^2y+b_{12}x y^2+b_{03}y^3)^k} \nonumber \\
 &+ \sum_{j+k=5}q_{jk}(a_{10}x+ a_{01}y)^j (b_{10}x+ b_{01}y)^k + h.o.t. \label{Y1xy}
\end{align}
Comparing the coefficients of the terms $x^j y^k$ with $j+k=3$ in \eqref{X1XY} with \eqref{X1xy}
and  \eqref{Y1XY} with \eqref{Y1xy}, respectively, 
we have
\begin{align*}
A_{30} & = a_{30}+ p_{30}a_{10}^3 + p_{21}a_{10}^2b_{10} +p_{12}a_{10}b_{10}^2
+p_{03}b_{10}^3 - a_{10}p_{30} - a_{01}q_{30}; \\
A_{21} &=a_{21}+3p_{30}a_{10}^2a_{01}+p_{21}(a_{10}^2b_{01}+2a_{10}a_{01}b_{10})
+p_{12}(2a_{10}b_{10}b_{01}+a_{01}b_{10}^2) \\
&+ 3p_{03}b_{10}^2b_{01} -a_{10}p_{21}-a_{01}q_{21};  \\ 
A_{12} & =a_{12}+3p_{30}a_{10}a_{01}^2+p_{21}(a_{01}^2b_{10}+2a_{10}a_{01}b_{01})
+p_{12}(2a_{01}b_{10}b_{01}+a_{10}b_{01}^2) \\
&+ 3p_{03}b_{10}b_{01}^2 -a_{10}p_{12}-a_{01}q_{12}; \\
A_{03} & = a_{03}+ p_{30}a_{01}^3 + p_{21}a_{01}^2b_{01} +p_{12}a_{01}b_{01}^2
+p_{03}b_{01}^3 - a_{10}p_{03} - a_{01}q_{03}.
\end{align*}

\begin{align*}
B_{30} & = b_{30}+ q_{30}a_{10}^3 + q_{21}a_{10}^2b_{10} +q_{12}a_{10}b_{10}^2
+q_{03}b_{10}^3 - b_{10}p_{30} - b_{01}q_{30}; \\
B_{21} &=b_{21}+3q_{30}a_{10}^2a_{01}+q_{21}(a_{10}^2b_{01}+2a_{10}a_{01}b_{10})
+q_{12}(2a_{10}b_{10}b_{01}+a_{01}b_{10}^2) \\
&+ 3q_{03}b_{10}^2b_{01} -b_{10}p_{21}-b_{01}q_{21};  \\ 
B_{12} & =b_{12}+3q_{30}a_{10}a_{01}^2+q_{21}(a_{01}^2b_{10}+2a_{10}a_{01}b_{01})
+q_{12}(2a_{01}b_{10}b_{01}+a_{10}b_{01}^2) \\
& + 3q_{03}b_{10}b_{01}^2 -b_{10}p_{12}-b_{01}q_{12}; \\
B_{03} & = b_{03}+ q_{30}a_{01}^3 + q_{21}a_{01}^2b_{01} +q_{12}a_{01}b_{01}^2
+q_{03}b_{01}^3 - b_{10}p_{03} - b_{01}q_{03}.
\end{align*}
Note that $\deg(A_{jk})=\deg(B_{jk})=-1$ when $j+k=3$ 
since $\deg(a_{jk})=\deg(b_{jk})=0$ when $j+k=1$
and $\deg(p_{jk})=\deg(q_{jk})=-1$ when $j+k=3$.

Similarly, by comparing the coefficients of the terms $x^j y^k$ with $j+k=5$ 
in \eqref{X1XY} with \eqref{X1xy}
and  \eqref{Y1XY} with \eqref{Y1xy}, respectively,
one can represent $A_{jk}$ and $B_{jk}$, $j+k=5$ in terms of $a_{jk}$, $b_{jk}$,
$p_{jk}$,  $q_{jk}$, $j+k \le 5$ and $A_{jk}$, $B_{jk}$, $j+k=3$.
Moreover, $\deg(A_{jk})=\deg(B_{jk})=-2$ when $j+k=5$.

Now we assume both nonresonance assumptions 
(A1) $\lambda^4\neq 1$ and  (A2) $\lambda^6 \neq 1$.
As we have done when finding the first-step Birkhoff transform in the previous section, 
we consider the complex coordinate system $(Z,W)$ with $Z= X+iY$ and $W=X-iY$. 
Note that the transformation $(x,y) \to (X,Y)$ given in Eq.~\eqref{gene.5X} and \eqref{gene.5Y}
is obtained by specifying the fifth-order terms of
the transformation given in Eq.~\eqref{gene.3X} and \eqref{gene.3Y}.
In particular, the third-order expansion of the function $F$ stays the same:
\begin{align}
Z_1&= X_1+iY_1=
\lambda (Z + i\tau_1 Z^2 W
+ \sum_{j+k=5} C_{jk}Z^j W^k) + h.o.t., \label{tran.in.tau2Z} \\
W_1&= X_1 - iY_1 
= \bar \lambda (W - i\tau_1 W^2 Z +  \sum_{j+k=5} \bar C_{jk}W^j Z^k) + h.o.t..\label{tran.in.tau2W}
\end{align}
Note that $\deg(C_{jk})=-2$ for all $(j,k)$ with $j+k=5$.
For later convenience, we give an explicit formula for the coefficient $C_{32}$ of the term $Z^3W^2$ in \eqref{tran.in.tau2Z}:
\begin{align}
C_{32}
=\bar\lambda
&((10A_{50} -2iA_{41}  +2A_{32} -2iA_{23} +2 A_{14}-10iA_{05}) \nonumber \\
&+i(10B_{50} -2iB_{41}  +2B_{32} -2iB_{23} +2 B_{14}-10iB_{05})). \label{c-32} 
\end{align}

Since the coefficients of the fourth-order terms in the equation
$\frac{\pa(Z_1, W_1)}{\pa(Z,W)}=1$ vanish, we obtain that
\begin{align}
5C_{50}+\bar{C}_{14}=0, \quad
2C_{41} +\bar{C}_{23} =0, \quad
\tau_1^2 + C_{32} +\bar C_{32} =0. \label{jacobi5}
\end{align}
Denote the imaginary part of the term $C_{32}$ by $\tau_2$.
Then we have $C_{32}=-\frac{\tau_1^2}{2}+i\tau_2$.
As we will show later, this quantity $\tau_2$ is the second twist coefficient
of the function $F$ at the point $P$. Note that $\deg(\tau_2) =-2$.

Using the same method as in \S \ref{tau1.intermediate},
we can find a coordinate transformation 
\begin{align*}
Z'&=Z+p_5(Z,W) = Z +\sum_{j+k=5} D_{jk}Z^j W^k,\\
W'&=W+\bar p_5(W,Z)= W +\sum_{j+k=5} \bar{D}_{jk}W^j Z^k.
\end{align*}
such that all the fifth-order terms $C_{jk}'$ of the function $(Z', W') \to (Z_1', W_1')$
vanish except the term $C_{32}'$, which stays the same: $C_{32}'=C_{32}$.
That is, we set $D_{32}=0$, and
\begin{align}
D_{50}=\frac{C_{50}}{1-\lambda^4}, \quad 
D_{41}=\frac{C_{41}}{1-\lambda^2}, \quad 
D_{23}=\frac{C_{23}}{1- \overline{\lambda}^2}=-2\bar{D}_{41}, \\
D_{14}=\frac{C_{14}}{1- \overline{\lambda}^4}=-5\bar{D}_{50}, \quad 
D_{05}=\frac{C_{05}}{1- \overline{\lambda}^6}.
\end{align}
Then we can rewrite the equations \eqref{tran.in.tau2Z} and \eqref{tran.in.tau2W} as
\begin{align}
Z_1'&= \lambda (Z' + i\tau_1 (Z')^2 W' + C_{32}(Z')^3 (W')^2) + h.o.t., \label{norm.5Z} \\
W_1'&= \bar \lambda (W' - i\tau_1 (W')^2 Z' + \bar C_{32}(W')^3 (Z')^2)) + h.o.t.. \label{norm.5W}
\end{align}
Moreover, the function 
\begin{align*}
s_6(Z,W)=D_{50}Z^5W-\bar{D}_{50}Z W^5 
+\frac{1}{2}(D_{41}Z^4W^2 - \bar{D}_{41}Z^2W^4)
+\frac{1}{6}(D_{05}W^6 - \bar{D}_{05}Z^6)
\end{align*}
satisfies $\frac{\pa s_6}{\pa W}=p_5(Z,W)$ and $\frac{\pa s_6}{\pa Z}=-\bar p_5(W,Z)$,
and $g_6(X,Y):=\frac{1}{2}s_6(X+iY, X-iY)$ is a real homogeneous polynomial of the real variables
$X$ and $Y$.
Then using the generating function method as in \S \ref{tau1.symp}, we obtain a  symplectic 
transformation $(X,Y) \to (\cX, \cY)$ of the real variables $(X,Y)$ 
such that in the complex coordinate system
$(\cZ, \cW)=(\cX+ i \cY, \cX - i\cY)$,
the function $F$ takes the same form:
\begin{align}
\cZ_1&= \lambda (\cZ + i\tau_1 \cZ^2 \cW
+ (-\frac{\tau_1^2}{2}+i\tau_2)\cZ^3 \cW^2) + h.o.t.
=e^{i(\theta+\tau_1|\cZ|^2 +\tau_2|\cZ|^4)}\cZ +h.o.t., \\
\cW_1&= \bar \lambda (\cW - i\tau_1 \cW^2 \cZ
+ (-\frac{\tau_1^2}{2} -i\tau_2) \cW^3\cZ^2) + h.o.t.
=e^{-i(\theta+\tau_1|\cZ|^2 +\tau_2|\cZ|^4)}\cW +h.o.t.. 
\end{align}
Here we have used the fact that $C_{32}= -\frac{\tau_1^2}{2}+i\tau_2$.
This explains the geometric meaning of the quantity $\tau_2$
and why it is called the second twist coefficient of the map $F$ at the elliptic point $P$.

Note that $\tau_2$ is the imaginary part of $C_{32}$.
We simplify and obtain the following formula for $\tau_2$:
\begin{align}
\tau_2(F,P)=
&\frac{3 (7 R^2 - 8 R L + 2 L^2)}{256 R^2(R-L)\sqrt{L(2R-L)} }
-\frac{L^{1/2}(27 R^2 - 40 R L +  10 L^2)}{384 R(R-L)(2R-L)^{3/2}} R'' \nonumber \\
&+ \frac{L^{3/2} (31 R^2 - 36 R L +  6 L^2) }{768 (R-L)(2R-L)^{5/2}}(R'')^2
-\frac{L^{3/2}R}{192(2R-L)^{3/2}}R^{(4)}. \label{tau2.sym.rewrite}
\end{align}
This is  the second twist coefficient of the one-step map $F$ of the symmetric table 
at the elliptic point $P$.
We have completed the derivation of the formula \eqref{tau2.sym.F}.
Observe that $\tau_2$ is homogeneous and $\deg(\tau_2)=-2$.

We view $\tau_2$ as a rational function of $(L, R, R'', R^{(4)})$ on the open locus $\{L>0, R>0,  2R-L>0\}$, after introducing the positive square roots $\sqrt{2R-L}$ and $\sqrt{L}$. Then $\tau_2$ has a simple polar divisor along $L=R$. Note that the remainder of $\sqrt{L} R^2 (R-L) (2 R-L)^{5/2}\tau_2$ modulo the factor $L-R$
is $\frac{1}{12}R^4(R R''+3)^2$. It follows from combining with \eqref{c03LR} that
\begin{align*}
\tau_2=-\frac{3 R^4 c_{03}^2}{4 \sqrt{L}(2 R-L)^{5/2}(R-L)}+ \rho(L, R, R'', R^{(4)}),
\end{align*}
for a regular function $\rho$. 
In particular, it follows that
\begin{pro}\label{c03pole}
The second twist coefficient $\tau_2$ has a simple pole along $L=R$ on the locus $\{c_{03} \neq 0\}$. \end{pro}

\subsection{Twist coefficients of the periodic orbits along the minor axis of an ellipse}
\label{sec.Kol}
Let $\cE$ be an ellipse with eccentricity $e$,
and $\cH$ be a hyperbolic caustic confocal to $\cE$ with eccentricity $h$.
This caustic corresponds to a pair of circles
in the phase space around the periodic orbit $\cO_2$ along the minor axis.
These two circles are permuted under the billiard map $F$
and $F^2$ is a smooth circle diffeomorphism on each of these two circles.
It is known that the rotation number of $F^2$ on either one of the two invariant circles 
is given by
\begin{align}
\rho(F^2, \cH)=\frac{\sF(\frac{\pi}{2}-\frac{\delta}{2},k) - \sF(\frac{\delta}{2},k)}{\sF(\frac{\pi}{2},k)},
\label{kolo.hyp}
\end{align}
where $k=\frac{2\sqrt{h}}{1+h}$, $\delta=\arcsin\frac{e\cdot (h^2-1)}{h^2-e^2}$,
and $\ds \sF(\alpha, k) = \int_0^{\alpha} \frac{1}{\sqrt{1-k^2\sin^2\theta}} d\theta$
is the incomplete elliptic integral of the first kind.
See \cite{Kol85, Tab94} for more details.

We have shown that 
$\tau_1(F,P)=\frac{b}{2}$ for the periodic orbit along the minor axis
($x$-axis) of the ellipse $\frac{x^2}{b^2} + y^2=1$, $0<b<1$.
See Remark \ref{ellipse.resonance} for discussions about the resonance case $b=1/\sqrt{2}$.
Plugging $(R'', R^{(4)})=(3b - \frac{3}{b}, \frac{9}{b} - 6 b - 3 b^3)$
into \eqref{tau2.sym.F}, we have
\begin{align}
\tau_2(F,P)=\frac{3 b (3-2b^2)}{32\sqrt{1 - b^2}},
\end{align}
for $b\in (0,1)\backslash \{\frac{1}{\sqrt{2}}\}$. 
For clarity of the discussion we will assume $b\in (0, \frac{1}{\sqrt{2}})$.
The eigenvalue of $D_{P}F^2$ is given by
\begin{align}
\lambda &=e^{i\theta}
=2(1-\frac{L}{R})^2 -1 + i 2(1-\frac{L}{R})\sqrt{L(\frac{2}{R}-\frac{L}{R^2})} \nonumber \\
&=2(1-2b^2)^2 -1  + i4b(1-2b^2)\sqrt{1-b^2}.
\end{align}
It follows that $\theta=\arccos(2(1-2b^2)^2 -1) \in (0,\pi)$.
Then  the rotation number of $F^2$ on the invariant circle containing a point $\cZ$ is given by
\begin{align}
\hat \rho(F^2, \cZ) = \frac{1}{2\pi}(\theta+b|\cZ|^2 + 2\tau_2(F,P) |\cZ|^4) + h.o.t..
\label{app.rot}
\end{align}

Next we will compare the above two formulas \eqref{kolo.hyp} and \eqref{app.rot} 
about the rotation numbers for orbits tangent to hyperbolic caustics.
Let  $\gamma(t)=(b\cos t, \sin t)$ be the natural parametrization of the ellipse 
$\frac{x^2}{b^2} + y^2=1$,
and $P(t)$ be the point in the phase space whose trajectory
moves along the normal direction $n(t)=(-\cos t, -b\sin t)$ 
at  the base point $\gamma(t)=(b\cos t, \sin t)$.
Then the trajectory of $P(t)$
is tangent to a hyperbolic caustic $\cH$ with eccentricity $h=\frac{1}{|\sin t|}$.
In the following we assume $0<t<\eps$.
It follows from \eqref{kolo.hyp}
that the rotation number of $F^2$ on the $F^2$-invariant circle containing the point 
$P(t)=(\gamma(t), n(t))$ is 
\begin{align}
\rho(F^2, t):=
\rho(F^2, P(t))=\frac{\sF(\frac{\pi}{2}-\frac{\delta(t)}{2},k(t)) - \sF(\frac{\delta(t)}{2},k(t))}{\sF(\frac{\pi}{2},k(t))},
\label{kolo-t}
\end{align}
where $\ds k(t)=\frac{2\sqrt{\sin t}}{1+\sin t}$, 
$\ds \delta(t)=\arcsin\frac{\sqrt{1-b^2}\cdot (1-\sin^2 t)}{1-(1-b^2)\sin^2 t}$.
The arc-length parameter $s$ of the ellipse is given by  $s(t)=\sE(t, \sqrt{1-b^2})$,
where $\sE(\alpha, k)=\int_0^{\alpha}\sqrt{1-k^2\sin^2\theta} d\theta$ is 
the  incomplete elliptic integral of the second kind.
On the other hand, the $(s,u)$-coordinates of the phase point $P(t)$
is $(s(t),0)$. It follows that $(x,y)=(\eta^{-1}s(t),  0)$,
\begin{align}
(X,Y)&=(x+p_{30}x^3 + p_{50}x^5, q_{30}x^3 + q_{50}x^5) + h.o.t.\nonumber \\
&=(\eta^{-1}s +p_{30}\eta^{-3}s^3, q_{30}\eta^{-3}s^3)+ o(t^4),\\
|\cZ|^2 &=\cX^2 + \cY^2= X^2 + Y^2 + o(t^4) 
=\eta^{-2}s(t)^2 +2p_{30}\eta^{-4}s(t)^4  + o(t^4);\\
\hat\rho(F^2, t) &:=\hat\rho(F^2, P(t)) 
= \frac{1}{2\pi}(\theta+b\eta^{-2}s(t)^2 +2(p_{30}b+\tau_2(F,P))\eta^{-4}s(t)^4) 
  + o(t^4).\label{rho.tw-t}
\end{align}
Comparing \eqref{kolo-t} with \eqref{rho.tw-t},
we observe they coincide at least to fourth-order:
\begin{align*}
\rho(F^2, 0) =\hat\rho(F^2, 0) &=\frac{\theta}{2\pi}, \\
\rho^{(j)}(F^2, 0) =\hat\rho^{(j)}(F^2, 0) &=0, \quad j=1,3,\\
\rho^{(2)}(F^2, 0) =\hat\rho^{(2)}(F^2, 0) &=\frac{1}{\pi} b \sqrt{1 - b^2},\\
\rho^{(4)}(F^2, 0) =
\hat\rho^{(4)}(F^2, 0)&=\frac{1}{4\pi} b (17 - 18 b^2) \sqrt{1 - b^2}.
\end{align*}

\section{Twist Coefficients for Periodic Orbits of Asymmetric Billiards}\label{asym.table}

In this section we consider the general case that the two functions 
$a(t)=\sum_{n\ge 1} a_{2n}t^{2n}$ and $b(t)=\sum_{n\ge 1} b_{2n}t^{2n}$ are different.
We will assume $a_2 \ge b_2>0$  without loss of generality.
Let $R_j(s)$ be the radius of curvature along the curve $\gamma_j$, where $s$ is the arc-length
parameter, $j=0, 1$. Note that  $R_j(s)$ is an even function, $j=0,1$,
and $R_0=\frac{1}{2a_2}\le R_1=\frac{1}{2b_2}$.

Let $\cO_2=\{P,F(P)\}$ be the periodic orbit of period 2 that bouncing back and forth
between $\gamma_0(0)$ and $\gamma_1(0)$.
Then the tangent matrix of $F^2$ at $P$ is given by
\begin{align}
D_{P}F^2
=\begin{bmatrix}\frac{L}{R_1}-1 & L \\ \frac{L}{R_0R_1}- \frac{1}{R_0} - \frac{1}{R_1} & \frac{L}{R_0}-1 \end{bmatrix}
\begin{bmatrix}\frac{L}{R_0}-1 & L \\ \frac{L}{R_0R_1}- \frac{1}{R_0} - \frac{1}{R_1} & \frac{L}{R_1}-1 \end{bmatrix} 
=:\begin{bmatrix} \ta_{10} & \ta_{01} \\ \tb_{10} & \tb_{01} \end{bmatrix}, 
\end{align}
where
\begin{align}
\ta_{10} =\tb_{01} &=2(\frac{L}{R_0}-1)(\frac{L}{R_1}-1) -1 
= 2L(\frac{L}{R_0R_1}- \frac{1}{R_0} -\frac{1}{R_1}) +1; \\
\ta_{01} &=2L(\frac{L}{R_1}-1); \\
\tb_{10} &=2(\frac{L}{R_0}-1)(\frac{L}{R_0R_1}- \frac{1}{R_0} - \frac{1}{R_1}).
\end{align}
It follows that the orbit $\cO_2$ is hyperbolic if $L\in (R_0, R_1) \cup(R_0+R_1, \infty)$,
is parabolic if $L\in \{R_0, R_1, R_0+R_1\}$,
and is elliptic if $L\in(0,R_0)\cup (R_1, R_0+R_1)$.
In the following we assume that $\cO_2$ is elliptic.  
Let $\eta=(- \ta_{01}/\tb_{10})^{1/4}$, which is a positive number. 
Then the complex vector $v=\begin{bmatrix} i\eta \\ \eta^{-1} \end{bmatrix}$
is an eigenvector of the matrix $D_{P}F^2$ for the eigenvalue 
\begin{align}
\lambda= \ta_{10}  - i\eta^{-2} \ta_{01}= \ta_{10} +i\eta^2 \tb_{10}.
\end{align}

For convenience, we introduce two frequently used notations:
\begin{align}
\Delta &=L(R_0 -L)(R_1-L)(R_0 +R_1 -L); \\
\Gamma &=(R_0 -L)(R_1-L) - L(R_0 +R_1 -L).
\end{align}

After the coordinate change $(s,u)\to (x,y)= (s/\eta, \eta u)$, we have
\begin{align}
a_{10} =b_{01} =\frac{\Gamma}{R_0R_1};\quad
a_{01}=\eta^{-2} \ta_{01} = -\frac{2\sqrt{\Delta}}{R_0R_1}; \quad
b_{10}=\eta^{2} \tb_{10}  = \frac{2\sqrt{\Delta}}{R_0R_1} = - a_{01}.
\end{align}
In particular, the eigenvalue $\lambda$ satisfies 
$\lambda= a_{10}  - i a_{01}= a_{10} +i b_{10}$.

Following the same approach as in \S \ref{sec.tau1},
we get the  first twist coefficient $\tau_1$ of the iterate $F^2$ at $P$
under the nonresonance assumption (B1) that $\lambda^4\neq 1$:
\begin{align}
\tau_1(F^2, P) &= \frac{1}{8} \Big(\frac{R_0+R_1}{R_0R_1}
-\frac{L}{R_0 +R_1 - L}\Big(\frac{R_1 -L}{R_0 -L}R_0'' + \frac{R_0 -L}{R_1 -L}R_1''\Big) \Big).
\end{align}
This is a special case  of the formula  on Page 303 of \cite{KP05}. 
Note that $\deg(\tau_1)=-1$.

The approach given in \S \ref{sec.tau2} can be used to calculate $\tau_2(F^2, P)$
for asymmetric tables.
However,  it does not lead to an explicit formula for $\tau_2(F^2, P)$. 
We obtain the formula \eqref{tau2.asym.F}  for $\tau_2(F^2, P)$ under an ansatz of the decomposition
of $\tau_2$ into quadratic terms of $R_{j}''$, $j= 0,1$, 
and linear terms of $R_{j}^{(4)}$, $j= 0,1$. See \eqref{tau2.ansatz} for more details.
Then we find an explicit formula \eqref{tau2.asym.F}  for $\tau_2$ 
by finding the coefficients of each terms.
Finally, we verify that this is indeed the correct formula for the second twist coefficient $\tau_2$.

To obtain the Birkhoff normal form containing both twist coefficients $\tau_1$ and $\tau_2$, 
it is necessary to assume both $\lambda^4\neq 1$ and $\lambda^6\neq 1$. 
However, only the assumption $\lambda^4\neq 1$ is used when computing the quantity $\tau_2$.
It follows that the potential poles of $\tau_2$ is determined by $\lambda^4=1$.
Observe that 
\begin{enumerate}
\item $\lambda=1$ when $L(R_0+R_1-L)=0$,

\item $\lambda=-1$ when $(R_0 -L)(R_1 -L)=0$,

\item $\lambda^2=1$ when $\Delta=L(R_0 -L)(R_1 -L)(R_0+R_1 -L)=0$

\item $\lambda^2 =-1$ when $\Gamma =(R_0 -L)(R_1 -L)- L(R_0+R_1 -L)=0$.
\end{enumerate}

Consider the case when $\gamma_0$ is fixed and $\gamma_1$ is flat. 
In this case, the dynamical  billiard on the table bounded by
$\gamma_0$ and $\gamma_1$ is  equivalent to the dynamical billiard on the table bounded by 
$\gamma_0$ and $\hat\gamma_0$, where $\hat\gamma_0$ is the reflection of $\gamma_0$
with respect to the line $\gamma_1$. More precisely, $\hat L = 2L$ and 
$\hat F(s,u)=(-s_2, -u_2)$,
where $(s_2, u_2) = F^2(s,u)$.
Then we substitute $L$ by $2L$ in \eqref{tau2.sym.rewrite} and obtain
\begin{align}
\tau_2(F^2, P) & =\frac{3 (7 R^2 - 16 R L + 8 L^2)}{512 R^2 (R - 2 L) \sqrt{L(R-L)}}
-\frac{L(27 R^2 - 80 R L +  40 L^2)}{768 R (R - 2 L) \sqrt{L(R-L)}(R - L)} R'' \nonumber \\
&+ \frac{L^{2} (31 R^2 - 72 R L +  24 L^2) }{1536 (R - 2 L) \sqrt{L(R-L)}(R - L)^{2}}(R'')^2
-\frac{L^{2} R}{192 \sqrt{L(R-L)}(R - L)}R^{(4)}.
\label{eq: tau_2 R_1 infty}
\end{align}

Another special case is when $\gamma_0$ and $\gamma_1$ are symmetric
as we have done in \S\ref{sec.tau2}.
Since we are considering the twist coefficients along the orbit $\cO_2$,
we multiply \eqref{tau2.sym.rewrite} by $2$ and obtain
\begin{align}
\tau_2(F^2, P) 
&=\frac{3 (7 R^2 - 8 R L + 2 L^2)}{128 R^2 (R -  L) \sqrt{L(2R-L)}}
-\frac{L(27 R^2 - 40 R L +  10 L^2)}{192 R (R - L) \sqrt{L(2R-L)}(2R - L)} R'' \nonumber \\
&+ \frac{L^{2} (31 R^2 - 36 R L +  6 L^2) }{384 (R - L) \sqrt{L(2R-L)}(2R - L)^{2}}(R'')^2
-\frac{L^{2} R}{96 \sqrt{L(2R-L)}(2R - L)}R^{(4)}. \label{eq.tau2.sym}
\end{align}

In the following we will consider the case that $0< L <\min\{R_0, R_1\}$.
We will explain the changes needed for the case $\max\{R_0, R_1\}< L < R_0 +R_1$
in Remark \ref{largeL}.
By analyzing the terms involved in the definition of $\tau_2$, we have that 
$\tau_2$ depends linearly on the fourth-order derivatives $R_0^{(4)}$ and  $R_1^{(4)}$
and quadratic on the second-order derivatives  $R_0^{(2)}$ and  $R_1^{(2)}$.
Moreover, the poles in \eqref{eq: tau_2 R_1 infty} and \eqref{eq.tau2.sym} are 
from the poles of the general case $\tau_2(F^2, P)$ by specializing to $R_1\rightarrow\infty$ and $R_0^{(2i)}=R_1^{(2i)}, i=0,1,2$, respectively.
Therefore, we make the following ansatz on the formulation of $\tau_2$, which has the minimal order of poles for each coefficient of the monomials in $R_0'', R_1'',R_0^{(4)}, R_1^{(4)}$:
\begin{align}
\tau_2&=\frac{1}{\sqrt{\Delta}}\Big(
\frac{3N(L,R_0,R_1)}{512R_0^2R_1^2(2(R_0-L)(R_1 -L)-R_0R_1)} \nonumber \\
&+ \frac{P(L,R_0,R_1)(R_0'')^2+ P(L,R_1,R_0)(R_1'')^2 +Q(L,R_0,R_1)R_0''R_1''}
{1536(R_0-L)^2(R_1 -L)^2(R_0 +R_1 -L)^2(2(R_0-L)(R_1 -L)-R_0R_1)} \nonumber  \\
&- \frac{R_1S(L,R_0,R_1)R_0''+ R_0S(L,R_1,R_0) R_1''}
{768R_0R_1(R_0-L)(R_1 -L)(R_0 +R_1 -L)(2(R_0-L)(R_1 -L)-R_0R_1)} \nonumber  \\
& -\frac{(R_1-L)T(L,R_0,R_1)R_0^{(4)}+ (R_0-L)T(L,R_1,R_0) R_1^{(4)}}
{192(R_0-L)(R_1 -L)(R_0 +R_1 -L)}\Big), \label{tau2.ansatz}
\end{align}
where 
$N(L,R_0,R_1)$, $P(L,R_0,R_1)$, $Q(L,R_0,R_1)$, $S(L,R_0,R_1)$ and $T(L,R_0,R_1)$
are polynomials with rational coefficients, 
whose degrees are determined by the facts that
$\deg L=1$, $\deg R_{j}^{(k)}=1-k$ and $\deg \tau_2=-2$.
It follows that
\begin{align}
\deg N=6,
\deg P=\deg Q=10,
\deg S=7,
\deg T=5.
\end{align}
Moreover, $Q$ is symmetric about $R_0$ and $R_1$: $Q(L, R_0, R_1)= Q(L, R_1, R_0)$.

\begin{remark}\label{largeL}
In the case $\max\{R_0, R_1\}< L < R_0 +R_1$, the formulation for $\tau_2$ is 
essentially the same as in the case $0< L <\min\{R_0, R_1\}$ with one
difference: the factor $\sqrt{\Delta}$ in \eqref{tau2.ansatz} should be replaced by $-\sqrt{\Delta}$.
This follows from matching the signs of the term $R-2L$ in \eqref{eq: tau_2 R_1 infty},
the term $R-L$ in \eqref{eq.tau2.sym}, respectively.  
\end{remark}

\subsection{The polynomial $N$}
Let $\tau_2^{(0)}(L,R_0,R_1)$ be the term that does not involve the
derivatives $R_{j}''$ and $R_{j}^{(4)}$, $j=0,1$ in \eqref{tau2.ansatz}.
That is,
\begin{align}
\tau_2^{(0)}(L,R_0,R_1)=&\frac{1}{\sqrt{\Delta}}
\frac{3N(L,R_0,R_1)}{512R_0^2R_1^2(2(R_0-L)(R_1 -L)-R_0R_1)},
\label{eq.N.def}
\end{align}
where $N(L, R_0, R_1)$ is a degree $6$ homogeneous polynomial that is symmetric in $R_0, R_1$. 
We consider two cases:
\begin{enumerate}
\item Setting $R_0=R$ and letting $R_1\rightarrow\infty$ in \eqref{eq.N.def}, we have
\begin{align*}
\lim_{R_1\rightarrow\infty}\tau_2^{(0)}
=&\lim_{R_1\rightarrow\infty}
\frac{3N(L,R_0,R_1)}{512R^2R_1^4\sqrt{L(R -L)}(R-2L)}
=\frac{3 (7 R^2 - 16 R L + 8 L^2)}{512 R^2 (R - 2 L) \sqrt{L(R-L)}},
\end{align*}
by comparing with \eqref{eq: tau_2 R_1 infty}.
It follows that  
\begin{align}
N(L, R_0, R_1)=R_1^4(7R_0^2-16R_0L+8L^2)+ l.o.t.,
\label{eq.tau20.R1-inf}
\end{align}
where l.o.t. stands lower order terms according to their $R_1$-degrees.

\item  Setting $R_0=R_1=R$ and comparing \eqref{eq.N.def} with \eqref{eq.tau2.sym}, we have 
\begin{align}
N(L, R, R)=&4(7R^2-8RL+2L^2)R^2(2(L-R)^2-R^2) \nonumber \\
=&28R^6-144R^5L+192R^4L^2-96R^3L^3+16R^2L^4. \label{eq.tau20.R}
\end{align}
\end{enumerate}

Since $N(L, R_0, R_1)$ is symmetric with respect to $R_0$ and $R_1$
and with highest degree of $L$ and $R_j$, $j=0,1$ bounded by $4$, we can write $N$ as
\begin{align}
N(L,R_0,R_1)=
\sum_{1\le j \le 3} \alpha_{j} R_0^j  R_1^j L^{6-2j}
+\sum_{0\leq i < j\leq 4:\; 2\le i+j \le 6} \beta_{ij}(R_0^iR_1^j+R_0^iR_1^j)L^{6-i-j}.
\label{eq.N.ab}
\end{align}
It follows from \eqref{eq.tau20.R1-inf} that $\beta_{24}=7$, $\beta_{14}=-16$, $\beta_{04}=8$.
Combining with \eqref{eq.tau20.R}, we have $\alpha_{3}=14$,
$\beta_{23}=-56$,
\begin{align}
\alpha_{2} + 2\beta_{13} = 176; \quad
\beta_{03}+ \beta_{12} = -48; \quad
\alpha_{1}+2\beta_{02}=16.
\end{align}
So we only need to find three unknown coefficients $(\beta_{02}, \beta_{12}, \beta_{13})$.
These unknowns can be found by solving a system of linear equations of \eqref{eq.N.ab} 
using three values of $N(L, R_0, R_1)$ given by \eqref{eq.N.def}. Then we get
\begin{align}
(\beta_{02}, \beta_{12}, \beta_{13})=(8, -32, 48).
\end{align}
Plugging them into \eqref{eq.N.ab} and simplifying it, we get
\begin{align}
N(L, R_0, R_1)
=&
8 L^4 (R_0^2 + R_1^2) - 16 L^3 (R_0^3 + 2 R_0^2 R_1 + 2 R_0 R_1^2 + R_1^3) \nonumber  \\
&+ 8 L^2 (R_0 + R_1)^2 (R_0^2 + 4 R_0 R_1 + R_1^2) \nonumber \\
& - 8 L R_0 R_1 (2 R_0^3 + 7 R_0^2 R_1 + 7 R_0 R_1^2 + 2 R_1^3)+ 
 7 R_0^2 R_1^2 (R_0 + R_1)^2 .  \label{eq.N.final}
\end{align}

\subsection{The polynomials $P$ and $Q$}\label{sec.PQ} 
Let $\tau_2^{(2,2)}(L,R_0,R_1, R_0'', R_1'')$ be the collection of
terms involving $(R_0'')^2$, $(R_1'')^2$ and $R_0''R_1''$ in \eqref{tau2.ansatz}.
That is,
\begin{align}
\tau_2^{(2,2)}
=\frac{P(L, R_0, R_1)(R_0'')^2 + P(L, R_1, R_0)(R_1'')^2 + Q(L,R_0,R_1) R_0'' R_1''}{1536\sqrt{\Delta}(R_0-L)^2(R_1-L)^2(R_0+R_1-L)^2[2(R_0-L)(R_1-L)-R_0R_1]}.
\label{eq.PQ.def}
\end{align}

Consider the case that $R_0=R_1=R$ and $R_0''=R_1''=R''$.
Comparing with \eqref{eq.tau2.sym}, we have
\begin{align*}
\tau_2^{(2,2)}=
&\frac{(2P(L, R, R)+Q(L, R, R))(R'')^2}{1536\sqrt{L(2R-L)}(R-L)^{5}[2(R-L)^2-R^2](2R-L)^{2}}\\
=& \frac{L^{2} (31 R^2 - 36 R L +  6 L^2) }{384 (R - L)\sqrt{L(2R-L)} (2R - L)^{2}}(R'')^2.
\end{align*}
It follows that
\begin{align}\label{eq: tau_2, 2,2, R equal}
2P(L, R, R)+Q(L, R, R)=4L^2 (31 R^2 - 36 R L +  6 L^2)(R-L)^4[2(R-L)^2-R^2].
\end{align}

\noindent\textbf{\ref{sec.PQ}.I. The polynomial $Q$.}
Recall that $Q(L, R_0, R_1)$ is a degree $10$ homogeneous polynomial in $L, R_0, R_1$
and is symmetric in $R_0, R_1$. 
It is easy to see that 
\begin{align}
\frac{Q(L,R_0,R_1)}{\phi(L,R_0,R_1)}
=&\tau_2(L,R_0,R_1,1,1) +\tau_2^{(0)}(L,R_0,R_1) \nonumber \\
&- \tau_2(L,R_0,R_1,1,0) - \tau_2(L,R_0,R_1,0,1), \label{eq.iso.Q}
\end{align}
where
\begin{align*}
\phi(L,R_0,R_1)=1536\sqrt{\Delta}(R_0-L)^2(R_1-L)^2(R_0+R_1-L)^2[2(R_0-L)(R_1-L)-R_0R_1].
\end{align*}
Suppose $\gamma_1(t)$ is given by the graph of the function $b(t)=L-b_2t^2 - b_4 t^4$. 
Then $R_1=\frac{1}{2b_2}$ and $R_1''=6b_2 - \frac{6b_4}{b_2^2}$.
In order to have $R_1''=1$, it suffices to assume $b_4=(b_2- \frac{1}{6})b_2^2$.
Letting $R_1\to\infty$, we have  $b_2\to 0$ and
the curve $\gamma_1$ becomes a flat boundary. It follows that
\begin{align}
\lim_{R_{1}\to\infty}\tau_2(L,R_0,R_1,1,1)
&=\lim_{R_{1}\to\infty}\tau_2(L,R_0,R_1,1,0);\\
\lim_{R_{1}\to\infty}\tau_2(L,R_0,R_1,0,1)
&=\lim_{R_{1}\to\infty}\tau_2^{(0)}(L,R_0,R_1).
\end{align}
Therefore,  the degree of $R_1$ in $Q(L,R_0,R_1)$ is strictly less than 
the degree of $R_1$ in $\phi(L,R_0,R_1)$.
It follows that  the degree of $R_1$ in $Q(L,R_0,R_1)$ is less than or equal to $5$.

We divide the task of finding $Q$ into two steps. First we consider the case that 
$R_0=R_1=R$ and write $Q(L,R,R)=\sum_{j=0}^{8} \alpha_{j} R^j L^{10 - j}$
with  undetermined coefficients $\alpha_j$, $0\le j \le 8$, based on the formula
\eqref{eq: tau_2, 2,2, R equal}.
These unknowns can be found by solving the linear equations 
using  nine values of $Q(L, R, R)$ given by \eqref{eq.iso.Q}:
\begin{align*}
\alpha_{0}=\alpha_{1}=0,
\alpha_{2}=-64,
\alpha_{3}= 384,
\alpha_{4}= -930,
\alpha_{5}= 1160,
\alpha_{6}= -780,
\alpha_{7}= 264,
\alpha_{8}= -34. 
\end{align*}
Plugging them into the expression of $Q$ and simplifying it, we get
\begin{align}
Q(L, R, R) &=-2 L^2 (L - R)^4 R^2 (32 L^2 - 64 L R + 17 R^2). \label{Q.sym}
\end{align}

Next we find the general formula for  $Q(L, R_0, R_1)$.
Since $Q(L, R_0, R_1)$ is symmetric with respect to $R_0$ and $R_1$ with degrees
of $R_0$ and $R_1$ less than or equal to $5$, 
we assume
\begin{align}
Q(R_0, R_1,L)
= &L^8 (c_ {20} (R_0^2 + R_1^2) + c_ {11} R_0 R_1) 
            + L^7 (c_ {30} (R_0^3 + R_1^3) +   c_ {21} (R_0^2 R_1 + R_1^2 R_0)) \nonumber \\
&  + L^6 (c_ {40} (R_0^4 + R_1^4) +   c_ {31} (R_0^3 R_1 + R_1^3 R_0) + 
                  c_ {22} R_0^2 R_1^2) \nonumber \\
 &  +  L^5 (c_ {50} (R_0^5 + R_1^5)  +      c_ {41} (R_0^4 R_1 + R_1^4 R_0)  
              +  c_ {32}(R_0^3 R_1^2 + R_1^3 R_0^2)) \nonumber \\
    &+    L^4 (  c_ {51} (R_0^5 R_1 + R_1^5 R_0)  
         +     c_ {42}(R_0^4 R_1^2 + R_1^4 R_0^2) +     c_ {33}R_0^3 R_1^3) \nonumber \\
  &   +   L^3 (  c_ {52}(R_0^5 R_1^2 + R_1^5 R_0^2) + c_ {43}(R_0^4 R_1^3 + R_1^4 R_0^3))  \nonumber \\
  &   +  L^2 (   c_ {53}(R_0^5 R_1^3 + R_1^5 R_0^3)  +    c_ {44}R_0^4 R_1^4).
 \label{eq.Q.general}
 \end{align}
Comparing \eqref{eq.Q.general} with \eqref{Q.sym}, we get 
\begin{align*}
2c_{20} + c_{11} & = -64, \\
2c_{30} + 2c_{21} & = 384, \\
2c_{40} + 2c_{31} + c_{22} & = -930, \\
2c_{50} + 2c_{41} + 2 c_{32} & =1160, \\
 2c_{51} + 2 c_{42} + c_{33}& = -780, \\
 2 c_{52} + 2 c_{43}& = 264, \\
 2 c_{53} + c_{44}& = -34.
\end{align*}
There are 10 independent coefficients among these 17 unknown coefficients. 
We  find ten data of $Q(L, R_0, R_1)$ using \eqref{eq.iso.Q}
and then produce a system of ten linear equations from \eqref{eq.Q.general}.
Then we get
\begin{align}
(c_{20}, c_{30}, c_{40}, c_{31}, c_{50}, c_{41}, c_{51}, c_{42}, c_{52}, c_{53})
=(0, 0, 0, -192, 0, 64, 0, -162, 0, 0).
\end{align}
Plugging them into \eqref{eq.Q.general} and simplifying it, we  get
\begin{align}
Q(L, R_0, R_1) 
&=- 2L^2 (L - R_0)^2(L-R_1)^2 R_0 R_1 (32 L^2 -32 L(R_0 +R_1) + 17 R_0 R_1).
\label{eq.Q.final}
\end{align}

\noindent\textbf{\ref{sec.PQ}.II. The polynomial $P$. }
It follows from \eqref{eq: tau_2, 2,2, R equal} and \eqref{Q.sym} that
\begin{align}
P(L, R, R) &= L^2 (L - R)^4 (24 L^4 - 192 L^3 R + 456 L^2 R^2 - 384 L R^3 + 79 R^4).
\label{eq.P.sym}
\end{align}

Setting $R_0=R, R_0''=R''$, $R_1''=0$ and letting $R_1\rightarrow\infty$ in \eqref{eq.PQ.def},  
we have
\begin{align}
\tau_2^{(2,2)}
=&\lim_{R_1\rightarrow\infty}\frac{P(L, R, R_1)}{1536R_1^6\sqrt{L(R-L)}(R-L)^2(R-2L)}(R'')^2  \nonumber \\
=&\frac{L^{2} (31 R^2 - 72 R L +  24 L^2) }{1536 (R - 2 L) \sqrt{L(R-L)}(R-L)^2}(R'')^2, \label{eq: tau_2, 2,2, R_1, infty} 
\end{align}
by comparing with \eqref{eq: tau_2 R_1 infty}. This implies
\begin{align}\label{eq: P, R_1 expansion}
P(L, R_0, R_1)= R_1^6(31R_0^2-72R_0L+24L^2)L^2+\sum_{j=5}^{10} R_1^{10-j}K_{j}(R_0, L).
\end{align}

Theoretically, we can use the undetermined coefficients method to find all the coefficients of 
the polynomials $K_{j}(R_0, L)$ in \eqref{eq: P, R_1 expansion}. 
However, it is difficult to stay away from possible degeneracy due to the large number of unknowns. We will manipulate the polynomial $P(L,R_0, R_1)$
and divide the unknowns into different groups of much smaller sizes.
To shorten our notations, we will set $L=1$ and omit the letter $L$ for the moment. Note that
\begin{align*}
\tau_2(R_0,R_1,1,1) + \tau_2(R_0,R_1, -1,-1) 
=2\tau_2^{(0)}+ \frac{2P(R_0,R_1)+2P(R_1,R_0)+2Q(R_0,R_1)}{\phi(R_0, R_1)}.
\end{align*}
It follows that the symmetric part $SP(R_0, R_1)=\frac{P(R_0,R_1)+P(R_1,R_0)}{2}$
of the polynomial $P(R_0,R_1)$ is given by
 \begin{align}
 SP(R_0, R_1)
=\frac{1}{4}[\phi(R_0, R_1)(\tau_2(R_0,R_1,1,1) + \tau_2(R_0,R_1, -1,-1)
 -2\tau_2^{(0)}(R_0,R_1)) -2Q(R_0,R_1)].\label{eq.SP.def}
\end{align}

Let $U=\frac{R_0+R_1}{2}$ and $V=\frac{R_0 - R_1}{2}$. 
Then the symmetric part $SP(R_0, R_1)$ can be written as a polynomial
of $U$ and $V$. Moreover,  the terms with  odd powers of $V$ vanish
since $SP(U,V)$ is a symmetric polynomial of $R_0$ and $R_1$.
Note that when $R_0=R_1=R$, we have $U=R$ and $V=0$.
Then it follows from \eqref{eq.P.sym} that
\begin{align}
SP(U,0)=(1 - U)^4 (24 - 192 U + 456 U^2 - 384 U^3 + 79 U^4).
\end{align}
The general form of the function $SP(U, V)$ can be written as:
\begin{align}
SP(U,V) &=(1 - U)^4 (24 - 192 U + 456 U^2 - 384 U^3 + 79 U^4) \nonumber \\
&+(a_{60}+a_{61}U+a_{62}U^2 + a_{63}U^3 + a_{64}U^4 + a_{65}U^5 + a_{66}U^6)V^2  \nonumber \\
&+(a_{40}+a_{41}U+a_{42}U^2 + a_{43}U^3 + a_{44}U^4)V^4 \nonumber  \\
&+(a_{20}+a_{21}U+a_{22}U^2 )V^6 
+a_{0}V^8. \label{eq.SP.UV}
\end{align}
Note that the equation \eqref{eq.SP.UV} can be viewed as a polynomial of the single variable $V$, where the variable $U$ is treated as a parameter:
\begin{align}
SP(U,V) &=SP(U,0) +a_6(U)V^2 + a_4(U)V^4 + a_2(U)V^6 +a_{0}V^8,  \label{eq.SP.V} \\
a_2(U)&=a_{20}+a_{21}U+a_{22}U^2,  \label{eq.a2}  \\
a_4(U)&=a_{40}+a_{41}U+a_{42}U^2 + a_{43}U^3 + a_{44}U^4, \label{eq.a4} \\
a_6(U)&=a_{60}+a_{61}U+a_{62}U^2 + a_{63}U^3 
+ a_{64}U^4 + a_{65}U^5 + a_{66}U^6.  \label{eq.a6}
\end{align}
We form a system of linear equations from \eqref{eq.SP.V} using  
a set of four values of $SP(R_0,R_1)$ from \eqref{eq.SP.def} with the same $U=(R_0 +R_1)/2$, 
from which we  get a number for $a_0$ and one value for each $a_k(U)$, $k=2,4,6$. 
Then by varying $U$, we obtain linear equations of the coefficients: 
\eqref{eq.a2} for $a_{2k}$, $0\le k \le 2$,
\eqref{eq.a4} for $a_{4k}$, $0\le k \le 4$, and  
\eqref{eq.a6} for $a_{6k}$, $0\le k \le 6$, respectively.
Finally, we  find the unknown coefficients $a_{jk}$ by solving the corresponding linear equations:
\begin{enumerate}
\item $a_0 = -17$;

\item $(a_{20}, a_{21}, a_{22})= (114, -372, 220)$;

\item $(a_{40}, a_{41}, a_{42}, a_{43}, a_{44})= (-41, 548, -1410, 1196, -310)$;

\item $(a_{60}, a_{61}, a_{62}, a_{63}, a_{64}, a_{65}, a_{66})
= (-216, 1152, -2222, 1784, -402, -124, 28)$.
\end{enumerate}

Next we find the antisymmetric part $AP(R_0, R_1)=\frac{P(R_0,R_1)-P(R_1,R_0)}{2}$ 
of the polynomial $P(R_0, R_1)$.
Note that
\begin{align}
\frac{2P(R_0,R_1)}{\phi(R_0, R_1)}
=\tau_2(R_0,R_1,1,0) + \tau_2(R_0,R_1, -1,0) -2\tau_2^{(0)}(R_0,R_1).
\end{align}
It follows that
\begin{align}
AP(R_0,R_1)=\frac{\phi(R_0, R_1)}{2}(\tau_2(R_0,R_1,1,0) + \tau_2(R_0,R_1, -1,0) 
-\tau_2^{(0)}(R_0, R_1))  -SP(R_0,R_1). \label{eq.AP.def}
\end{align}
Note that the polynomial $AP(R_0, R_1)$ can be written as a polynomial
of $U=\frac{R_0+R_1}{2}$ and odd powers of $V=\frac{R_0 - R_1}{2}$:
\begin{align}
AP(U,V) &=
(b_{70}+b_{71}U+b_{72}U^2 + b_{73}U^3 + b_{74}U^4 
+ b_{75}U^5 + b_{76}U^6+ b_{77}U^7)V \nonumber \\
&+(b_{50}+b_{51}U+b_{52}U^2 + b_{53}U^3 + b_{54}U^4 + b_{55}U^5)V^3 \nonumber  \\
&+(b_{30}+b_{31}U+b_{32}U^2 +b_{33}U^3)V^5
+(b_{10}+b_{11}U)V^7. \label{eq.AP.b}
\end{align}
Note that the terms with  even powers of $V$ vanish
since $AP(U,V)$ is an antisymmetric polynomial of $R_0$ and $R_1$.
Using the same method as we have done for the symmetric part $SP(U, V)$,
we view \eqref{eq.AP.b} as a polynomial of $V$ with coefficients 
$b_{k}(U)$, $k=1,3,5,7$, depending on the parameter $U$.
By varying $V$ with fixed $U$,
we get  the coefficients $b_{k}(U)$, $k=1,3,5,7$ using the values of $AP(R_0, R_1)$ given by \eqref{eq.AP.def}.
The unknowns $b_{jk}$ can be found by solving the systems of linear equations
using the values of $b_{k}(U)$ that we have just found:
\begin{enumerate}
\item $(b_{10}, b_{11}) = (-20, -28)$;

\item $(b_{30}, b_{31}, b_{32}, b_{33}) = (220, -660, 556, -164)$; 

\item $(b_{50}, b_{51}, b_{52}, b_{53}, b_{54}, b_{55}) 
= (-336, 2016, -4472, 4584, -2204, 412)$;

\item $(b_{70}, b_{71}, b_{72}, b_{73}, b_{74}, b_{75}, b_{76}, b_{77}) 
= (0, -288, 1872, -4896, 6556, -4692, 1668, -220)$.
\end{enumerate}

Collecting the symmetric and the antisymmetric parts and simplifying it, we have
\begin{align}
P(R_0,R_1)
=(R_1-1)^4 (&48 R_0^3 (R_1-2) + 24 (R_1-1)^2 - 72 R_0(R_1-1)(R_1-2) \nonumber \\
  & + R_0^2 (216 - 216 R_1 + 31 R_1^2)). 
\end{align}
Since $P$ is a homogeneous polynomial of degree $10$, we have
\begin{align}
P(L,R_0,R_1)
=L^2(R_1-L)^4 (&48 R_0^3 (R_1-2L) + 24 L^2(R_1-L)^2 - 72L R_0(R_1-L)(R_1-2L) \nonumber \\
  & + R_0^2 (216L^2 - 216 R_1L + 31 R_1^2)). \label{eq.P.final}
\end{align}

\subsection{The polynomial $S$}
Let $\tau_2^{(2)}$ be the collection of terms depending linearly on $R_0''$ and $R_1''$
in \eqref{tau2.ansatz}.
That is, 
\begin{align}
\tau_2^{(2)}=&-\frac{S(L, R_0, R_1) R_1 R_0''
+S(L, R_1, R_0)R_0 R_1''}{768\sqrt{\Delta}R_0 R_1(R_0-L)(R_1-L)(R_0+R_1-L)[2(R_0-L)(R_1-L)-R_0R_1]},
\label{eq.S.def}
\end{align}
where $S(L, R_0, R_1)$ is a degree $7$ homogeneous polynomial in $L, R_0, R_1$ (not necessarily symmetric in $R_0, R_1$). 
Consider two special cases:

(1) Setting $R_0=R$, $R_0''=R''$, $R_1''=0$ and letting $R_1\rightarrow\infty$
in \eqref{eq.S.def}, 
we get  that
\begin{align*}
\tau_2^{(2)}&=\lim_{R\rightarrow\infty}-\frac{1}{768}\frac{S(L, R, R_1)}{R R_1^4\sqrt{L(R-L)}(R-L)(R-2L)}R'' \\
&=-\frac{L(27 R^2 - 80 R L +  40 L^2)}{768 R \sqrt{L(R-L)}(R - 2 L) (R - L)} R''
\end{align*}
by comparing with \eqref{eq: tau_2 R_1 infty}.
This implies
\begin{align}\label{eq: S, R_1, expansion}
S(L, R_0, R_1)= R_1^4L(27 R_0^2 - 80 R_0 L +  40 L^2)
+ l.o.t.,
\end{align}
where l.o.t. means terms whose $R_1$-degree is less than or equal to $3$.

(2) Setting $R_0=R_1=R$ and $R_0''=R_1''=R''$, we get that
\begin{align*}
\tau_2^{(2)}&
=-\frac{1}{384}\frac{S(L, R, R)}{\sqrt{L(2R-L)}(R-L)^3(2R-L)(2(R-L)^2-R^2)R} R'' \\
&=-\frac{L(27 R^2 - 40 R L +  10 L^2)}{192 R \sqrt{L(2R-L)}(R - L) (2R - L)} R''
\end{align*}
 by comparing \eqref{eq.S.def} with \eqref{eq.tau2.sym}.
This implies
\begin{align}\label{eq: S, R, R, L}
S(L, R, R)=2L(27 R^2 - 40 R L +  10 L^2)(R-L)^2[2(R-L)^2-R^2].
\end{align}
We divide the task into two steps by finding the symmetric part $SS(L, R_0, R_1)$ and the anti-symmetric part $AS(L, R_0, R_1)$ of  the polynomial $S(L, R_0, R_1)$, respectively.

In the following we assume $L=1$, $R^{(4)}_{j}=0$, $j=0,1$ 
and introduce the short notation $\tau_2(R_0,R_1,R_0'',R_1'')$ for the corresponding
second twist coefficient of the orbit $\cO_2$. 
Similarly, we set $\Phi(R_0, R_1)=-768\sqrt{\Delta}(R_0-1)(R_1-1)(R_0+R_1-1)[2(R_0-1)(R_1-1)-R_0R_1]$ and obtain
\begin{align*}
S(R_0, R_1)&=(\tau_2(R_0,R_1,1,0)-\tau_2^{(0)}(R_0,R_1)-\tau_2^{(22)}(R_0,R_1,1,0))\Phi(R_0, R_1)R_0, \\
S(R_1, R_0)&=(\tau_2(R_0,R_1, 0,1)-\tau_2^{(0)}(R_0,R_1)-\tau_2^{(22)}(R_0,R_1,0,1))\Phi(R_0, R_1)R_1.
\end{align*}
So the symmetric part $SS(R_0,R_1)=\frac{S(R_0,R_1)+S(R_1,R_0)}{2}$ of the polynomial $S$ is given by
\begin{align}
SS(R_0,R_1)=\Big(
&(\tau_2(R_0,R_1,1,0)-\tau_2^{(0)}(R_0,R_1)-\tau_2^{(22)}(R_0,R_1,1,0))R_0 \nonumber \\
+&(\tau_2(R_0,R_1, 0,1)-\tau_2^{(0)}(R_0,R_1)-\tau_2^{(22)}(R_0,R_1,0,1))R_1
\Big)\frac{\Phi(R_0, R_1)}{2}.
\end{align}

We rewrite \eqref{eq: S, R, R, L} in the coordinates 
$(U, V)=(\frac{R_0+R_1}{2},\frac{R_0-R_1}{2})$ with undetermined coefficients:
\begin{align}
SS(U,V)&=2(27U^2 - 40 U +  10)(U-1)^2(2(U-1)^2-U^2) \nonumber \\
&+(a_{50}+a_{51}U+a_{52}U^2 + a_{53}U^3 + a_{54}U^4 + a_{55}U^5)V^2 \nonumber  \\
&+(a_{30}+a_{31}U+a_{32}U^2 +a_{33}U^3)V^4
+(a_{10}+a_{11}U)V^6.
\end{align}
Using the same method in finding the polynomial $SP$, we find that
\begin{align}
SS(U,V)&=2 (27 U^2 - 40 U + 10) (U - 1)^2 (2 (U - 1)^2 - U^2) \nonumber  \\
&+ (-152 +  600 U - 822 U^2 + 488 U^3 - 108 U^4) V^2 
+ (48 - 84 U +     54 U^2) V^4.
\end{align}

Next we turn to the antisymmetric part $AS(R_0,R_1)=\frac{S(R_0,R_1)-S(R_1,R_0)}{2}$,
which is given by
\begin{align}
AS(R_0,R_1)=\Big(
&(\tau_2(R_0,R_1,1,0)-\tau_2^{(0)}-\tau_2^{(22)}(1,0))R_0 \nonumber  \\
- &(\tau_2(R_0,R_1, 0,1)-\tau_2^{(0)}-\tau_2^{(22)}(0,1))R_1
\Big)\frac{\Phi(R_0, R_1)}{2}.
\end{align}
We assume 
\begin{align}
AS(U,V)
&=(b_{60}+b_{61}U+b_{62}U^2 + b_{63}U^3 + b_{64}U^4 + b_{65}U^5 + b_{66}U^6)V \nonumber  \\
&+(b_{40}+b_{41}U+b_{42}U^2 + b_{43}U^3 + b_{44}U^4)V^3
+(b_{20}+b_{21}U+b_{22}U^2 )V^5
+b_0V^7.
\end{align}
Using  the same method in finding the polynomial $AP$, we find that
\begin{align}
AS(U,V)
=&(-144U+552U^2 -714U^3 +360U^4 -54U^5 )V \nonumber  \\
&+(-104+ 394U -400U^2 +108U^3)V^3
+(40-54)V^5.
\end{align}
Collecting the symmetric and antisymmetric parts and simplifying it, we have 
\begin{align}
S(R_0, R_1)=&(R_1 -1 )^2 (40 ( R_1 -1)^2 + 3 R_0^3 (9 R_1-16)  \nonumber \\
&  -  80 R_0 (2 - 3 R_1 + R_1^2) + 3 R_0^2 (56 - 56 R_1 + 9 R_1^2)).
\end{align}
Recall that $S(L,R_0, R_1)$ is a homogeneous polynomial of degree $7$. It follows that
\begin{align}
S(L, R_0, R_1)=&L(R_1 -L)^2 (40 ( R_1 -L)^2L^2 + 3 R_0^3 (9 R_1-16L)  \nonumber \\
&  -  80 R_0 (2L^2 - 3 R_1L + R_1^2)L + 3 R_0^2 (56L^2 - 56 R_1L + 9 R_1^2)).
\end{align}

\subsection{The polynomial $T$}
Let $\tau_2^{(4)}$ be the collection of
terms involving $R_0^{(4)}$ and $R_1^{(4)}$ in \eqref{tau2.ansatz}.
That is, 
\begin{align}
\tau_2^{(4)}=&-\frac{1}{192\sqrt{\Delta}}(\frac{T(L, R_0, R_1)}{(R_0+R_1-L)(R_0-L)}R_0^{(4)}+\frac{T(L, R_1, R_0)}{(R_0+R_1-L)(R_1-L)}R_1^{(4)}),
\label{eq.T.def}
\end{align}
where $T(L, R_0, R_1)$ is a degree $5$ homogeneous polynomial, which is not necessarily symmetric in $R_0, R_1$. Let $\tau_2(L,R_0,R_1,R_0^{(4)},R_1^{(4)})$ be the second twist coefficient of the orbit $\cO_2$
with $R_{j}''=0$, $j=0,1$. It follows that
\begin{align}
T(L, R_0, R_1) =-192\sqrt{\Delta}(R_0+R_1-L)(R_0-L)
\cdot (\tau_2(L,R_0,R_1,1,0) - \tau_2^{(0)}). \label{eq.T.data}
\end{align}

Consider the two special cases:

(1) Setting $R_0=R_1=R, R_0^{(4)}=R_1^{(4)}=R^{(4)}$ in \eqref{eq.T.def}, we have that
\begin{align}
&\tau_2^{(4)}
=-\frac{T(L, R, R)}{96\sqrt{L(2R-L)}(R-L)^2(2R-L)}R^{(4)}
=- \frac{L^{2} R}{96 \sqrt{L(2R-L)}(2R - L)}R^{(4)}, \nonumber
\end{align}
 by comparing with \eqref{eq.tau2.sym}.
It follows that
\begin{align}
T(L, R, R)=(R-L)^2L^2R.  \label{eq: T, RRL}
\end{align}

(2) Setting $R_0=R$, $R_{0}^{(4)}=R^{(4)}$, $R_1^{(4)}=0$ 
and letting $R_1\rightarrow\infty$ in \eqref{eq.T.def}, 
we have that 
\begin{align*}
\tau_2^{(4)}
=\lim_{R_1\rightarrow\infty}-\frac{1}{192}\frac{T(L, R, R_1)}{R_1^2\sqrt{L(R -L)}(R -L)}R^{(4)}
=-\frac{L^{2} R}{192 \sqrt{L(R -L)}(R - L)}R^{(4)},
\end{align*}
by comparing with \eqref{eq: tau_2 R_1 infty}.
It follows that 
\begin{align}
\label{eq: R_1 infty, T} T(L, R_0, R_1)=R_1^2(L^2R_0)+R_1\cdot G_4(R_0, L)+G_5(R_0, L),
\end{align}
where $G_4$ and $G_5$ are homogeneous polynomial in $R_0, L$ of degree $4$ and $5$ respectively.
Comparing \eqref{eq: T, RRL}  and \eqref{eq: R_1 infty, T}, we get 
\begin{align*}
&(R-L)^2L^2R=R^3L^2+R\cdot G_4(R, L)+G_5(R,L),
\end{align*}
or equally,
\begin{align*}
 -2R^2L^3+RL^4=R\cdot G_4(R, L)+G_5(R,L).
\end{align*}
This implies that $G_5(R,L)$ is divisible by $R$, i.e. $G_5(R_0, L)=R_0 H_4(R_0, L)$. 
Now we have 
\begin{align}
\label{eq: sum G_4, H_4}
G_4(R,L)+H_4(R,L)=-2RL^3+L^4.
\end{align}
Suppose that
\begin{align*}
G_4(R,L)=\sum\limits_{j=0}^4a_jR^jL^{4-j},\quad
H_4(R,L)=\sum\limits_{j=0}^4b_jR^jL^{4-j},
\end{align*}
with undetermined coefficients $a_j$ and $b_j$, $0\le j \le 4$.
Then \eqref{eq: sum G_4, H_4} implies that
\begin{align*}
&a_0+b_0=1,\quad a_1+b_1=-2,\\
&a_j+b_j=0, \quad 2\leq j\leq 4. 
\end{align*}
These unknowns can be found by solving the linear equations of \eqref{eq: R_1 infty, T} using the values of $T(L, R_0, R_1)$ given by \eqref{eq.T.data}:
\begin{align*}
(a_0, a_1, a_2, a_3, a_4) = (0, -2, 0, 0, 0). 
\end{align*}
Plugging them into \eqref{eq: R_1 infty, T} and simplifying it, we get
\begin{align}
T(L, R_0, R_1)=L^2R_0(R_1-L)^2. \label{eq.T.final}
\end{align}

\section*{Acknowledgments}
The authors would like to thank Richard Moeckel for helpful discussions.
The authors are very grateful to the anonymous referees for many comments and suggestions,
which have helped them improve the presentation of the paper greatly.

\end{document}